\documentclass[11pt,english]{article}

\usepackage[T1]{fontenc}
\usepackage[latin9]{inputenc}
\usepackage[a4paper]{geometry}
\usepackage{babel}
\usepackage{amsmath}
\usepackage{amsthm}
\usepackage{amssymb}
\usepackage{color}
\usepackage{setspace}
\usepackage[unicode=true,pdfusetitle,
 bookmarks=true,bookmarksnumbered=false,bookmarksopen=false,
 breaklinks=false,pdfborder={0 0 0},pdfborderstyle={},backref=false,colorlinks=false]
 {hyperref}

\makeatletter

\usepackage[nameinlink,capitalise,noabbrev]{cleveref}

\theoremstyle{plain}
\newtheorem{theorem}{Theorem}[section]
\crefname{theorem}{Theorem}{Theorems}
\theoremstyle{plain}
\newtheorem{lemma}[theorem]{Lemma}
\crefname{lemma}{Lemma}{Lemmas}
\theoremstyle{plain}
\newtheorem{corollary}[theorem]{Corollary}
\theoremstyle{plain}
\newtheorem*{claim*}{Claim}
\theoremstyle{definition}
\newtheorem{definition}[theorem]{Definition}
\theoremstyle{plain}
\newtheorem{conjecture}[theorem]{Conjecture}
\theoremstyle{plain}

\theoremstyle{definition}

\theoremstyle{definition}

\theoremstyle{plain}

\usepackage{tikz}
\usetikzlibrary{external}
\tikzstyle{vertex}=[circle,draw=black,fill=black,inner sep=0,minimum size=0.2cm,text=white,font=\footnotesize]

\date{}

\crefformat{equation}{#2(#1)#3}

\let\originalleft\left
\let\originalright\right
\renewcommand{\left}{\mathopen{}\mathclose\bgroup\originalleft}
\renewcommand{\right}{\aftergroup\egroup\originalright}

\usepackage{verbatim}

\makeatletter
\renewcommand*{\UrlTildeSpecial}{%
  \do\~{%
    \mbox{%
      \fontfamily{ptm}\selectfont
      \textasciitilde
    }%
  }%
}%
\let\Url@force@Tilde\UrlTildeSpecial
\makeatother

\usepackage{bbm}

\usepackage[labelfont=bf,labelsep=period]{caption}

\makeatother

\begin{document}

\title{Hypergraph cuts above the average}

\author{David Conlon\thanks{Mathematical Institute, Oxford OX2 6GG, United Kingdom. Email: \href{david.conlon@maths.ox.ac.uk} {\nolinkurl{david.conlon@maths.ox.ac.uk}}. Research supported by a Royal Society University Research Fellowship and by ERC Starting Grant 676632.}
\and 
Jacob Fox\thanks{Department of Mathematics, Stanford University, Stanford, CA 94305. Email: \href{fox@math.mit.edu} {\nolinkurl{jacobfox@stanford.edu}}. Research supported by a Packard Fellowship and by NSF Career Award DMS-1352121.}
\and 
Matthew Kwan\thanks{Department of Mathematics, Stanford University, Stanford, CA 94305. Email: \href{mailto:mattkwan@stanford.edu} {\nolinkurl{mattkwan@stanford.edu}}. This research was done while the author was working at ETH Zurich.}
\and 
Benny Sudakov\thanks{Department of Mathematics, ETH, 8092 Z\"urich, Switzerland. Email: \href{mailto:benjamin.sudakov@math.ethz.ch} {\nolinkurl{benjamin.sudakov@math.ethz.ch}}.
Research supported in part by SNSF grant 200021-175573.}}

\maketitle
\global\long\def\QQ{\mathbb{Q}}
\global\long\def\RR{\mathbb{R}}
\global\long\def\E{\mathbb{E}}
\global\long\def\Var{\operatorname{Var}}
\global\long\def\Cov{\operatorname{Cov}}
\global\long\def\spec{\operatorname{spec}}
\global\long\def\CC{\mathbb{C}}
\global\long\def\NN{\mathbb{N}}
\global\long\def\ZZ{\mathbb{Z}}
\global\long\def\GG{\mathbb{G}}
\global\long\def\supp{\operatorname{supp}}
\global\long\def\one{\mathbbm{1}}
\global\long\def\range#1{\left[#1\right]}
\global\long\def\d{\operatorname{d}\!}
\global\long\def\falling#1#2{\left(#1\right)_{#2}}
\global\long\def\im{\operatorname{im}}
\global\long\def\rank{\operatorname{rank}}
\global\long\def\sign{\operatorname{sign}}
\global\long\def\mod{\operatorname{mod}}
\global\long\def\id{\operatorname{id}}
\global\long\def\tr{\operatorname{tr}}
\global\long\def\adj{\operatorname{adj}}
\global\long\def\Unif{\operatorname{Unif}}
\global\long\def\Po{\operatorname{Po}}
\global\long\def\Bin{\operatorname{Bin}}
\global\long\def\Ber{\operatorname{Ber}}
\global\long\def\Geom{\operatorname{Geom}}
\global\long\def\Exp{\operatorname{Exp}}
\global\long\def\Hom{\operatorname{Hom}}
\global\long\def\vol{\operatorname{vol}}
\global\long\def\floor#1{\left\lfloor #1\right\rfloor }
\global\long\def\ceil#1{\left\lceil #1\right\rceil }
\global\long\def\cond{\,\middle|\,}
\global\long\def\cI{\mathcal{I}}
\global\long\def\cJ{\mathcal{J}}
\global\long\def\cA{\mathcal{A}}
\global\long\def\cB{\mathcal{B}}
\global\long\def\cF{\mathcal{F}}
\global\long\def\cG{\mathcal{G}}
\global\long\def\cV{\mathcal{V}}
\global\long\def\Ggood#1{G_{#1}^{\mathrm{good}}}
\global\long\def\Gbad#1{G_{#1}^{\mathrm{bad}}}
\global\long\def\Rgood#1{R_{#1}^{\mathrm{good}}}
\global\long\def\Rbad#1{R_{#1}^{\mathrm{bad}}}
\global\long\def\mono{\mathrm{mono}}
\global\long\def\multi{\mathrm{multi}}
\global\long\def\Hpart#1{H_{#1}^{\mathrm{part}}}
\global\long\def\Zpart{Z^{\mathrm{part}}}
\global\long\def\zpart{z^{\mathrm{part}}}
\global\long\def\Gpart#1{G_{#1}^{\mathrm{part}}}
\global\long\def\epart#1{e_{#1}^{\mathrm{part}}}

\begin{abstract}
An $r$-cut of a $k$-uniform hypergraph $H$ is a partition
of the vertex set of $H$ into $r$ parts and the size of the cut
is the number of edges which have a vertex in each part.
A classical result of Edwards says that every $m$-edge graph  
has a 2-cut of size $m/2 + \Omega(\sqrt{m})$ and this is best possible.
That is, there exist cuts which exceed the expected size 
of a random cut by some multiple of the standard deviation.
We study analogues of this and related results in hypergraphs.  
First, we observe that similarly to graphs, every $m$-edge $k$-uniform hypergraph has an $r$-cut 
whose size is $\Omega(\sqrt m)$ larger than the expected size of a random $r$-cut. 
Moreover, in the case where $k=3$ and $r=2$ this bound is best possible and is 
attained by Steiner triple systems.
Surprisingly, for all other cases (that is, if $k \geq 4$ or $r \geq 3$),
we show that every $m$-edge $k$-uniform hypergraph has an $r$-cut whose size is $\Omega(m^{5/9})$ larger than the expected size of a random $r$-cut. 
This is a significant difference in
behaviour, since the amount by which the size of the largest cut exceeds the
expected size of a random cut is now considerably larger than the standard deviation.
\end{abstract}

\section{Introduction}

The \emph{max-cut} of a graph $G$ is the maximum number of edges
in a bipartite subgraph of $G$. Equivalently, it is the maximum
\emph{size} of a \emph{cut}, where a cut of a graph is a partition
of its vertex set into two parts and the size of such a cut is the
number of edges with one vertex in each part.
The max-cut parameter has been studied extensively over the last 50 years, 
both from the algorithmic perspective
emphasised in computer science and optimisation and from the extremal perspective
taken in combinatorics. For a thorough (though now somewhat outdated)
overview of the subject, see the survey of Poljak and Tuza \cite{PTsurvey}.

In computer science, the problem of computing the max-cut of a graph 
already featured in Karp's famous list of 21 NP-complete problems~\cite{Karp}. 
However, efficient algorithms were subsequently found for computing the max-cut of certain 
restricted classes of graphs, including planar graphs~\cite{Hadlock, OD}. 
More recently, the problem has attracted attention due to the 
approximation algorithm of Goemans and Williamson~\cite{GW} and its 
connections to the unique games conjecture~\cite{KKMO}.

In combinatorics, it is an important problem to estimate the max-cut of a graph
in terms of its number of edges. As a first approximation, it is easy to see that the max-cut
of any $m$-edge graph is at least $m/2$. Indeed, this is the expected
size of a uniformly random cut. Moreover, though there are certain classes of graphs
for which the max-cut is significantly larger (see, for example, \cite{KKV}),
the max-cut of an $m$-edge graph is typically quite close to $m/2$. As such,
much of the focus has been on maximising the \emph{excess} of a cut, defined to 
be its size minus $m/2$.

There are several interesting results about the excess of a graph with $m$ edges.
Answering a conjecture of Erd\H os, Edwards \cite{Edw73} proved
that every $m$-edge graph has a cut with excess at least
$\left(\sqrt{8m+1}-1\right)/8$. This bound is exact for complete
graphs and so best possible when $m$ is of the form $\binom{n}{2}$. Nevertheless, 
answering another conjecture of Erd\H os, Alon \cite{Alo96}
showed that for infinitely many $m$ this bound can
be improved by $\Omega\left(m^{1/4}\right)$. We refer the reader to~\cite{AH98,BS02}
for further results in this direction and to~\cite{Alo96,ABKS03,AKS05,Erd79,PT94,She92} 
for improved bounds when the graph is known to be $H$-free.

For graphs without isolated vertices, it is also interesting to ask
for bounds on the excess in terms of the number of vertices. Edwards
\cite{Edw73} proved that every connected $n$-vertex graph has a
cut with excess at least $\left(n-1\right)/4$. Erd\H os, Gy\'arf\'as
and Kohayakawa \cite{EGK97} gave a simplified proof of this fact
(and of the Edwards bound) and
they also proved that every graph with no isolated vertices has a
cut with excess at least $n/6$. Both of these bounds are best possible, attained
by a clique and a disjoint union of triangles, respectively.

The goal of this paper is to extend these classical results, where possible,
from graphs to hypergraphs. We now describe our results in detail.

\subsection{Max-cut for hypergraphs}

A \emph{hypergraph} $H = (V, E)$ consists of a vertex set $V$ and a collection $E$ of subsets
of $V$ known as the edges of $H$. A $k$-\emph{uniform hypergraph}
(or \emph{$k$-graph }for short) is a hypergraph where all of the edges have size
exactly $k$. In particular, a 2-graph is simply a graph.

There are several possible notions of max-cut for hypergraphs that
generalise the notion of max-cut for graphs. One obvious generalisation
is to define a cut of a hypergraph $H$ to be a partition of its vertex
set into two parts and the size of such a cut to be the number
of edges which have nonempty intersection with each of the two parts.
As in graphs, we then define the max-cut of $H$ to be the maximum
size of a cut. An equivalent definition of
the same notion is that the max-cut of $H$ is the size of its largest
2-colourable subgraph (see, for example, \cite{Lov}). In computer science,
the problem of computing this type of max-cut is commonly called max
set splitting and the restriction to $k$-uniform hypergraphs
is called max E$k$-set splitting (see, for example, \cite{Gur, Has}). 

A second very natural generalisation to $k$-graphs $H$ is to
define a cut to be a partition of the vertices into $k$ parts and
the size of such a cut to be the number of edges with exactly
one vertex in each part. The corresponding notion of max-cut can also be
defined as the size of the largest $k$-partite subgraph
of $H$ (note that $k$-partite $k$-graphs have special significance
in combinatorics, for example, in Ryser's conjecture \cite{Aha} and 
in hypergraph Tur\'an theory \cite[Section~14]{FraChapter}). 

One can also interpolate between these two notions
of max-cut. For a $k$-graph $H$ and $2\le r\le k$, we define an
\emph{$r$-cut} of $H$ to be a partition of the vertex set of $H$ into $r$ parts
and the size of the $r$-cut to be the number of edges which have a
vertex from every part (we say such edges are \emph{multicoloured}).
Then the \emph{max-$r$-cut }of $H$ is the maximum size of an
$r$-cut. This generalised notion of hypergraph max-cut was first
considered by Erd\H os and Kleitman \cite{EK68}, who observed that
the expected size of a uniformly random $r$-cut is 
\[
\frac{S\left(k,r\right)r!}{r^{k}}m,
\]
where $S\left(k,r\right)$ is a Stirling number of the second
kind (the number of unlabelled partitions of $\left\{ 1,\dots,k\right\} $
into $r$ nonempty sets). For the convenience of the reader, we remark that $S\left(k,2\right)2!/2^{k}=1-2^{1-k}$ and $S\left(k,k\right)k!/k^{k}=k!/k^k$ are the constants corresponding to the aforementioned special cases of 2-cuts in $k$-graphs and $k$-cuts in $k$-graphs. Just as for graphs, we define the \emph{excess} of an $r$-cut
in a $k$-graph to be its size minus $\left(S\left(k,r\right)r!/r^{k}\right)m$.
We will chiefly be interested in proving Edwards-type bounds for $r$-cuts of a $k$-graph.

For fixed $2\le r\le k$, there are at least two natural conjectures
regarding the maximum excess of an $r$-cut of a $k$-graph. The first conjecture,
inspired by the graph case, is that if $m=\binom{n}{k}$, then complete
$k$-graphs have the smallest max-$r$-cut among all $m$-edge
$k$-graphs. This conjecture was proposed by Scott \cite[Problem~4.1]{Sco05}
in the case $r=2$. As a corollary of our \cref{thm:chromatic}, this would imply a 
lower bound of $\Omega\left(m^{\left(k-1\right)/k}\right)$ on the maximum excess
of an $m$-edge $k$-graph.

The second plausible conjecture is that, for any fixed $2\le r\le k$,
the smallest maximum $r$-cut over all $m$-edge $k$-graphs
has excess $\Theta\left(\sqrt{m}\right)$. There are several reasons
why this is a natural guess. First, one of the most basic and effective
means for proving bounds on max-cut in graphs is via the chromatic number,
and, as we will see in \cref{sec:chromatic}, these methods generalise
to prove a lower bound of $\Omega\left(\sqrt{m}\right)$ on the maximum
excess of a hypergraph cut. Second, the standard deviation of the size of a uniformly random
cut can be as small as $\Theta\left(\sqrt{m}\right)$. In particular, this is the case when $H$ is a \emph{linear} $k$-graph, meaning no pair of vertices
is involved in more than one edge. Third, we have already seen
that $\Theta\left(\sqrt{m}\right)$ is the correct order of magnitude
if $r=k=2$ and, moreover, we will next see that it is also the correct
order of magnitude when $r=2$ and $k=3$ (this disproves Scott's conjecture in the case $k=3$). A ($k$-uniform) \emph{perfect matching} is a $k$-graph in which every vertex appears in exactly one edge. A \emph{Steiner triple system} is a 3-graph in which every pair of vertices appears in exactly one edge. It is a very classical theorem in design theory~\cite{Kir47} that an $n$-vertex Steiner
triple system exists whenever $n\equiv1$ or $n\equiv3$ mod
6. 

\begin{theorem}
\label{thm:3-2}Every $m$-edge 3-graph
has a 2-cut with excess at least $\left(\sqrt{24m+1}-1\right)/16$. Every $n$-vertex connected 3-graph $H$
has a 2-cut with excess at least
$\left(n-1\right)/8$. Both of these bounds are best possible, being attained by any Steiner
triple system.

If $H$ is a 3-graph with $n$ vertices, none of which are isolated, then it has
a 2-cut with excess at least $n/12$. This bound is also best possible, being attained by a perfect matching.
\end{theorem}

\begin{proof}
It is easy to check (see \cite{Has}) that a 2-cut with size $z$ in an $m$-edge 3-graph
$H$ corresponds precisely to a 2-cut with size $2z$ in the $3m$-edge multigraph $G\left(H\right)$
obtained by replacing each edge of $H$ with a triangle on the same vertex
set. Each of the standard bounds for graph max-cut are also valid
for multigraphs, so the desired lower bounds immediately follow. Now, observe that if $H$ is a perfect matching then $G(H)$ is a disjoint union of triangles and if $H$ is a Steiner triple system then $G(H)$ is a complete graph. These are precisely the graphs which are extremal for the standard graph max-cut theorems.
\end{proof}

In light of \cref{thm:3-2}, it is tempting to study highly
regular sparse hypergraphs (such as those coming from design theory)
as potential examples of $k$-graphs whose maximum excess of an $r$-cut is $\Theta\left(\sqrt{m}\right)$.
However, surprisingly, our main theorem shows that one
can always find $r$-cuts of significantly greater excess, except when $r = 2$ and $k = 2$ or $3$.

\begin{theorem}
\label{thm:nontrivial} For any fixed $2\le r\le k$ with $r\ge3$
or $k\ge4$, every $m$-edge $k$-graph has an $r$-cut
with excess $\Omega\left(m^{5/9}\right)$.
\end{theorem}

We prove \cref{thm:nontrivial} in \cref{sec:nontrivial-reductions,subsec:3-cuts-almost-linear,subsec:2-cuts-almost-linear}. The idea
is to first randomly choose the parts for some of the vertices in our hypergraph $H$,
in such a way as to reduce the problem to finding a large 2-cut for a certain
random hypergraph derived from $H$. Although the edges of this random hypergraph
will not be independent, we will be able to reduce to the case where the
dependencies are in a certain sense sparse, meaning that a typical
small subset of vertices is not likely to witness any dependencies.
The idea then is to split the vertex set into many such subsets, to greedily find a good cut for each of these vertex subsets individually,
and then to combine these cuts randomly to give a cut of $H$ with large excess. We remark that the general idea of partitioning a random graph into subsets with few dependencies also appeared in~\cite{APS}.

Regarding upper bounds, we show that random $k$-graphs can
have smaller max-$r$-cut than complete graphs with the same number
of edges. This disproves Scott's conjecture for $k>3$. 

\begin{theorem}
\label{thm:upper}For $k>3$, in the binomial random $k$-graph $\GG_{k}\left(n,n^{3-k}\right)$,
a.a.s.\footnote{By \emph{asymptotically almost surely} or \emph{a.a.s.}, we mean that
the probability of the event is $1-o\left(1\right)$. Asymptotics will
always be as $n\to\infty$ or as $m\to\infty$ --- it should be clear
from context which is meant. The constants implied by all asymptotic notation may depend on fixed values of $k$ and $r$.}\ every $r$-cut has excess at most $O\left(n^{2}\right)=O\left(m^{2/3}\right)$.
\end{theorem}

As in graphs, it is also interesting to consider bounds for the
excess of a $k$-graph $r$-cut in terms of the number of vertices.
Using methods developed by Crowston, Fellows, Gutin, Jones,
Rosamond, Thomass\'e and Yeo \cite{CFGJRTY11}, Giannopoulou, Kolay
and Saurabh \cite{GKS12} proved that $n$-vertex $k$-graphs which
are partition connected (a stronger requirement than being
connected) have 2-cuts with excess at least $\left(n-1\right)/\left(k2^{k-1}\right)$.
We give improved bounds in a much more general setting.

\begin{theorem}
\label{thm:excess-n}For any fixed $2\le r\le k$, consider an $n$-vertex
$k$-graph $H$ with no isolated vertices. Then $H$ has an $r$-cut
with excess $\Omega\left(n\right)$. In particular, if $k>2$, then
$H$ has a 2-cut with excess at least $n/\left(k2^{k-1}\right)$.
This is best possible, being tight for a perfect matching.
\end{theorem}

We prove \cref{thm:excess-n} in \cref{sec:excess-n} by analysing a
variant of the \emph{method of conditional probabilities} introduced
by Erd\H{o}s and Selfridge \cite{ES}. We note that the proof easily
gives rise to a deterministic polynomial-time algorithm for constructing
cuts of the required size.

\subsection{Structure of the paper}

The rest of the paper is organised as follows. In \cref{sec:upper}, we prove \cref{thm:upper} and in \cref{sec:chromatic}, we give some simple lower bounds on max-$r$-cut based on the (strong) chromatic number. The bulk of the paper is then spent proving \cref{thm:nontrivial,thm:excess-n}. We remark that these two proofs are independent of each other.

The proof of \cref{thm:excess-n} is presented first, in \cref{sec:excess-n}. It is convenient to first show the required bound for max-2-cut, then use a simple reduction based on conditional expectations to deduce the general result for max-$r$-cut. Next, we turn to the proof of \cref{thm:nontrivial}. Again, we make some reductions and then prove the theorem in a few special cases. In \cref{sec:nontrivial-reductions}, we give some simple arguments to treat the case where $H$ is ``far'' from being a linear hypergraph and we give a few different arguments which allow us to reduce our consideration to the case $k=r=3$ and the case $r=2$. Then, in \cref{subsec:3-cuts-almost-linear}, we treat the case of 3-cuts in almost-linear 3-graphs, while in \cref{subsec:2-cuts-almost-linear} we treat the case of 2-cuts in almost-linear $k$-graphs. Both of these proofs follow the same basic strategy, but the former case is much simpler and can be viewed as a ``warm-up''.


\subsection{Notation}

We use standard graph-theoretic notation throughout. The vertex set
and edge set of a hypergraph $H$ are denoted $V\left(H\right)$ and $E\left(H\right)$,
respectively, and we write $e\left(H\right)=\left|E\left(H\right)\right|$
for the number of edges in $H$. For a subset $U$ of the vertex set of $H$, we write $H\left[U\right]$ for
the subgraph of $H$ \emph{induced }by $U$, that is, the subgraph of $H$ whose vertex set is $U$ and
whose edge set consists of all edges in $E(H)$ which are entirely contained in $U$. 

Less standardly, a \emph{mixed} $k$-graph is a hypergraph where all
edges have size at most $k$ and a \emph{multihypergraph} is a hypergraph
where multiple copies of each edge are allowed (formally, a multihypergraph
$H$ is defined by a vertex set $V\left(H\right)$ and a multiset
$E\left(H\right)$ of subsets of $V\left(H\right)$). We also define
$k$-multigraphs and mixed $k$-multigraphs, in the
obvious way. In addition, for a (multi)hypergraph $H$, we
let $G\left(H\right)$ be the multigraph obtained by replacing each
edge $e\in E(H)$ with a clique $K_{\left|e\right|}$.

For a real number $x$, the floor and ceiling functions are denoted
$\floor x=\max\{i\in\mathbb{Z}:i\le x\}$ and $\ceil x=\min\{i\in\mathbb{Z}:i\ge x\}$. We will however omit floor and ceiling symbols whenever they are not crucial, for the sake of clarity of presentation.
For real numbers $x,y$, we write $x\lor y$ to denote $\max\{x,y\}$
and $x\land y$ to denote $\min\{x,y\}$. All logs are base
$e$.

Finally, we use standard asymptotic notation throughout, as follows.
For functions $f=f(n)$ and $g=g(n)$, we write $f=O(g)$ to mean that there
is a constant $C$ such that $|f|\le C|g|$, $f=\Omega(g)$
to mean that there is a constant $c>0$ such that $f\ge c|g|$,
$f=\Theta(g)$ to mean that $f=O(g)$ and $f=\Omega(g)$, and 
$f=o(g)$ to mean that $f/g\to0$ as $n \to \infty$.

\section{\label{sec:upper}Upper bounds}

Let $K_{n}^{\left(k\right)}$ be the complete $k$-graph on $n$ vertices.
In this section, we prove upper bounds on the max-$r$-cut in $K_{n}^{\left(k\right)}$
and in random $k$-graphs, in the process proving \cref{thm:upper}.

\begin{lemma} \label{thm:Kn}
Every $r$-cut in $K_{n}^{\left(k\right)}$ has excess at most
$O\left(n^{k-1}\right)$.
\end{lemma}

\begin{proof}
Consider an $r$-cut with parts of size $n_{1},\dots,n_{r}$, where $\sum_{i}n_{i}=n$. The size of this $r$-cut is 
$$
\sum_{\left(s_{1},\dots,s_{r}\right)}\prod_{i=1}^{r}{n_{i} \choose s_{i}},
$$
where the sum runs over all $\left(s_{1},\dots,s_{r}\right)\in\left(\NN^{+}\right)^{r}$
with $\sum_{i}s_{i}=k$. Recall that the expected size of a random $r$-cut is 
$$\frac{S\left(k,r\right)r!}{r^{k}}\binom{n}{k}=\sum_{\left(s_{1},\dots,s_{r}\right)}\binom{k}{s_{1},\dots,s_{r}}\frac{1}{r^{k}}\binom{n}{k}.$$
Now, as in \cite{EK68}, one may verify that the max-$r$-cut is obtained when each $\floor{n/r}\le n_{i}\le\ceil{n/r}$ (we say such a cut is \emph{equitable}).
Therefore, the excess of the $r$-cut is 
\begin{align*}
\sum_{\left(s_{1},\dots,s_{r}\right)}\left(\prod_{i=1}^{r}{n_{i} \choose s_{i}}-\binom{k}{s_{1},\dots,s_{r}}\frac{1}{r^{k}}\binom{n}{k}\right)&=\sum_{\left(s_{1},\dots,s_{r}\right)}\left(\prod_{i=1}^{r}\frac{1}{s_{i}!}\right)\left(\prod_{i=1}^{r}n_{i}^{s_{i}}-\frac{n^{k}}{r^{k}}\right)+O\left(n^{k-1}\right)\\&=O\left(n^{k-1}\right),
\end{align*}
as desired.
\end{proof}

Now we can prove \cref{thm:upper}.
\begin{proof}[Proof of \cref{thm:upper}]
Let $p=n^{3-k}$ and $G\in\GG_{k}\left(n,p\right)$, so that
a.a.s.~$G$ has $p\binom{n}{k}+o\left(n^{2}\right)$ edges. Consider
any $r$-cut of $K_{n}^{\left(k\right)}$. By \cref{thm:Kn}, it has size
at most $\left(S\left(k,r\right)r!/r^{k}\right){n \choose k}+O\left(n^{k-1}\right)$,
so the size of the corresponding $r$-cut in $G$ is $p\left(S\left(k,r\right)r!/r^{k}\right){n \choose k}+O\left(pn^{k-1}+t\right)$
with probability $e^{-\Omega\left(t^{2}/pn^{k}\right)}$. If $t$
is a sufficiently large multiple of $n^{2}$ then this probability
is $o\left(r^{-n}\right)$, which allows us to take a union bound
over all $r$-cuts.
\end{proof}

\section{\label{sec:chromatic}Basic chromatic number lower bounds}

If an $r$-cut of a hypergraph $H$ is very ``imbalanced'' (for example,
if all but one part is empty), then its size is small. Therefore, a
natural way to obtain an $r$-cut above the average size is to choose a random
cut which is biased towards being ``balanced''. For example, we could choose
a uniformly random $r$-cut of $H$ into parts of equal sizes, but it turns
out that we can improve on this by choosing a random $r$-cut which is
balanced in a certain way with respect to a fixed strong colouring of $H$.

We recall that a \emph{strong colouring} of a (multi)hypergraph is a colouring in which
all the vertices in each edge receive a distinct colour. The
\emph{strong chromatic number} of a (multi)hypergraph is then the minimum number
of colours required for a strong colouring.

\begin{theorem}
\label{thm:chromatic}For any fixed $2\le r\le k$ and any $n$-vertex,
$m$-edge $k$-multigraph $H$ with strong chromatic number $\chi$,
there is an $r$-cut with excess $\Omega\left(m/\chi\right)$.
\end{theorem}

\begin{proof}
Consider the cut obtained by splitting the colour classes in a strong $\chi$-colouring of $H$ randomly into $r$
groups, each with equally many classes. (We can assume that $\chi$ is divisible by $r$, by adding at most $r-1$ additional isolated vertices each in their own colour class). We claim that the probability
an edge $e$ is multicoloured is 
\[
\frac{S\left(k,r\right)r!}{r^{k}}+\Omega\left(1/\chi\right),
\]
which clearly suffices to prove the theorem. To prove our claim, consider
some ordering on the $k$ vertices of $e$ (we assume the vertices
are in fact the numbers $\left\{ 1,\dots,k\right\} $) and, for each
$0\le i\le k$, let $\omega_{i}:e\to\left\{ 1,\dots,r\right\} $ be
the random (``partial'') cut where the parts of the first $i$ vertices
of $e$ are coloured according to our special ``balanced'' random
cut and the last $k-i$ vertices are coloured uniformly at random.
We will prove that for each $1\le i\le k$ the probability $e$
is multicoloured with respect to $\omega_{i}$ is no less than the
the corresponding probability for $\omega_{i-1}$, and that for $2\le i\le k$
these probabilities differ by $\Omega\left(1/\chi\right)$.

First note that $\omega_{0}$ and $\omega_{1}$ actually have the
same distribution. Next, note that $\omega_{i}$ and $\omega_{i-1}$
can be coupled to differ only on vertex $i$. Expose the value of
each $\omega_{i}\left(v\right)=\omega_{i-1}\left(v\right)$, for $v \neq i$.
With probability $\Omega\left(1\right)$, in $e$ we have exposed
a vertex in every part except some part $j$ (if this situation does
not occur, then it has already been determined whether $e$ is multicoloured
or not). Note that the conditional probability that $e$ is multicoloured
with respect to $\omega_{i-1}$ is $\Pr\left(\omega_{i-1}\left(i\right)=j\right)=1/r$,
whereas the corresponding probability for $\omega_{i}$ is 
\[
\Pr\left(\omega_{i}\left(i\right)=j\right)\ge\frac{\chi/r}{\chi-i+1}\ge\frac{1}{r}+\Omega\left(\frac{i-1}{\chi}\right).
\]
This proves the desired claim.
\end{proof}
As the strong chromatic number of $H$ is trivially at most $n$, we have the following immediate corollary.
\begin{corollary}
\label{cor:basic-m-n}For any fixed $2\le r\le k$ and any $n$-vertex,
$m$-edge $k$-multigraph $H$, there is an $r$-cut with excess $\Omega\left(m/n\right)$.
\end{corollary}
Next, the following corollary was mentioned in the introduction.
\begin{corollary}
\label{cor:sqrt-m}For any fixed $2\le r\le k$ and any $m$-edge $k$-multigraph $H$, there is an $r$-cut with excess $\Omega\left(\sqrt{m}\right)$.
\end{corollary}

\begin{proof}
Let $\chi$ be the strong chromatic number of $H$ and consider a
strong colouring with $\chi$ colours. Since every pair of distinct colours must
appear in an edge, $\binom{\chi}{2} \le m\binom{k}{2}$ and, hence, $\chi= O\left(\sqrt{m}\right)$.
\end{proof}

We end this section with a slight generalisation of \cref{cor:basic-m-n}. Recall that
$G\left(H\right)$ is the multigraph obtained by replacing each
edge $e\in E(H)$ with a clique $K_{\left|e\right|}$.

\begin{lemma}
\label{lem:advanced-m-n}For any fixed $2\le r\le k$, consider a
$k$-multigraph $H$. Suppose there is an $n'$-vertex subset $W$
inducing $m'$ edges of $G\left(H\right)$. Then $H$ has an $r$-cut of
excess $\Omega(m'/n')$.
\end{lemma}

\begin{proof}
Let $\omega$ be the random $r$-cut where the parts of the vertices outside $W$ are chosen uniformly at random and the parts of the vertices inside $W$ are chosen with a uniformly random
equipartition of $W$. Virtually the same proof as for \cref{thm:chromatic}
shows that for each of the $\Omega(m')$ edges $e$ of $H$ with $\left|e\cap W\right|\ge2$,
the probability that $e$ is multicoloured is $S\left(k,r\right)r!/r^{k}+\Omega\left(1/n'\right)$.
Indeed, for each $0\le i\le |e\cap W|$, let $\omega_i$ be the random ``partial'' cut where the first $i$ vertices of $e\cap W$ are coloured according to $\omega$, and the last $|e\cap W|-i$ are coloured uniformly at random. As in the proof of \cref{thm:chromatic}, the probability $e$
is multicoloured with respect to $\omega_{i}$ is no less than the
the corresponding probability for $\omega_{i-1}$, and for $2\le i\le |e\cap W|$
these probabilities differ by $\Omega\left(1/n'\right)$.

Also, all edges $e$ of $H$ with $\left|e\cap W\right|<2$ are multicoloured in $\omega$ with probability exactly $S\left(k,r\right)r!/r^{k}$. Therefore, the expected excess of $\omega$ is $\Omega(m'/n')$.
\end{proof}

\section{\label{sec:excess-n}Proof of \texorpdfstring{\cref{thm:excess-n}}{Theorem~\ref{thm:excess-n}}}

In this section, we prove \cref{thm:excess-n}. By first choosing which
vertices are in the last $r-2$ parts, we will be able to reduce
the problem of finding a large-excess $r$-cut in a $k$-graph to
the problem of finding a large 2-cut in a certain mixed $k$-graph.
Note that we can generalise the notion of excess to a mixed multihypergraph
$H$ in an obvious way: if $Z$ is the size of a uniformly random
$r$-cut of $H$ and an $r$-cut has size $z$, then we say the excess
of that $r$-cut is $z-\E Z$. \cref{thm:excess-n} will be a consequence
of the following lemma. 
\begin{lemma}
\label{thm:excess-W}For any $k\ge3$, consider an $n$-vertex mixed
$k$-multigraph $H$ with its vertices ordered. Let $W$ be the set
of vertices which, according to this ordering, are among the first
two vertices of some edge of $H$ which has size at least three. Then
$H$ has a 2-cut with excess at least $\left|W\right|/2^{k}$.
\end{lemma}

\begin{proof}
We assume that the vertex set $V$ consists of the integers $1,\dots,n$
with the natural ordering.

We greedily build a cut $\omega:V\to\left\{ 1,2\right\} $ using the
method of conditional probabilities, with the slight adaptation that
we temporarily mark some vertices as ``undetermined'' if their choice
of part has no effect on the subsequent conditional expectation. We
will determine the part of an undetermined vertex $v$ later, when
we reach a vertex that shares an edge with $v$. The details are as
follows.

Following our chosen order, we iteratively do the following for
each vertex $v$. Let $U$ be the set of undetermined vertices that
precede $v$. For each edge $e$, let $E_{e}$ be the event that $e$
is multicoloured (2-coloured), conditioning on the choices $\omega\left(v'\right)$
we have made so far for the determined vertices $v'$, and choosing
the parts $\omega\left(v''\right)$ of the other vertices $v''$ uniformly
at random. Let $Z_{v}=\sum_{e}\one_{E_{e}}$ be the size of this random
cut (so $Z_1$ is just the unconditional size of a uniformly random cut). Say that an edge $e$ is \emph{uncertain} if $\Pr\left(E_{e}\right)\notin\left\{ 0,1\right\}$
and let $U_{v}\subseteq U$ be the set of undetermined vertices $u\in U$
such that $u$ and $v$ share an uncertain edge. For each of the $2^{\left|U\right|+1}$
possible assignments of parts for the vertices in $U\cup\left\{ v\right\} $,
consider the conditional expectation of $Z_{v}$ given that assignment.
The average of these conditional expectations is $\E Z_{v}$. If all
these conditional expectations are exactly $\E Z_{v}$, mark $v$
as undetermined. Otherwise, for an assignment that maximizes this conditional
expectation, fix the parts of the vertices in $U_{v}\cup\left\{ v\right\} $
according to that assignment. The vertices
in $U_{v}$ are no longer undetermined. (We emphasise that we are considering an assignment of parts for all the vertices in $U\cup\left\{ v\right\} $, including the vertices in $U\setminus U_{v}$, but we are only fixing the parts of the vertices in $U_{v}\cup\left\{ v\right\} $.)  Finally, at the end of the
procedure, after going through all the vertices, choose the parts
for the remaining undetermined vertices arbitrarily.

Now we show that this procedure produces a cut with the desired excess.
First, we make the important observation that each of the conditional
expectations compared during the algorithm is an integer multiple
of $1/2^{k-1}$. Next, we prove the following claim.

\begin{claim*}
At each step $v$ of the process,
\begin{enumerate}
\item [(1)]there is an assignment $A$ of the vertices in
$U\cup\left\{ v\right\} $ such that $\E\left[Z_{v}\cond A\right]\ge\E Z_{v}+\left|U_{v}\right|/2^{k-1}$;
\item [(2)]for any assignment $A^v$ of the vertices in $U_{v}\cup \{v\}$ and any assignment $A^U$ of the vertices in $U\setminus U_v$, $\E\left[Z_{v}\cond A^v,A^U\right]=\E\left[Z_{v}\cond A^v\right]$.
\end{enumerate}
\end{claim*}

\begin{proof}
Consider some step $v$ and assume (1) and (2) held for all preceding
steps. First note that for any assignment $A$ of the vertices in $U$, we have
\setcounter{equation}{2}
\begin{equation}
\E\left[Z_{v}\cond A\right]=\E Z_{v}.\label{eq:undetermined-irrelevant}
\end{equation}
If $v-1$ is undetermined then this immediately follows from the definition of an undetermined vertex, and otherwise it follows from the fact that (2) held in step $v-1$
(note that no $u\in U$ can have
been in $U_{v-1}$ or otherwise it would have become determined). Note also that there is no uncertain
edge which involves more than one vertex of $U$, because the later
vertex would have become determined at its step by (1).

Now, to prove (2), note that any conditional expectation of $Z_{v}$
is the sum of the conditional probabilities of each edge being multicoloured. The
set of uncertain edges containing a vertex in $U\setminus U_{v}$
is disjoint from the set of uncertain edges containing a vertex in $U_v\cup \{v\}$, so we
can express $Z_{v}$ as the sum of two terms $Z_{v}^{U}$ and $Z_{v}^{v}$,
the former of which depends on $\left(\omega\left(v'\right)\right)_{v'\in U\setminus U_{v}}$
but not $\left(\omega\left(v'\right)\right)_{v'\in U_v\cup \{v\}}$, and the latter of which depends on
$\left(\omega\left(v'\right)\right)_{v'\in U_v\cup \{v\}}$ but not $\left(\omega\left(v'\right)\right)_{v'\in U\setminus U_{v}}$ (strictly speaking there is also a third term which doesn't depend on any vertex in $U$, but we can view this as being part of say $Z_{v}^{v}$).
Applying \cref{eq:undetermined-irrelevant}, for any assignment $A^U$
of the vertices in $U\setminus U_{v}$, we have
\[
\E Z_{v}^{U}+\E Z_{v}^{v}=\E Z_{v}=\E\left[Z_{v}\cond A^U\right]=\E\left[Z_{v}^{U}\cond A^U\right]+\E Z_{v}^{v}.
\]
So, $\E\left[Z_{v}^{U}\cond A^U\right]=\E Z_{v}^{U}$. Therefore, for
any $A^U$, $A^v$ as in the statement of (2),
\[
\E\left[Z_{v}\cond A^v,A^U\right]=\E\left[Z_{v}^{U}\cond A^U\right]+\E\left[Z_{v}^{v}\cond A^v\right]=\E Z_{v}^{U}+\E\left[Z_{v}^{v}\cond A^v\right]=\E\left[Z_{v}\cond A^v\right],
\]
proving (2). Finally, we prove (1) for step $v$. Since no edge contains multiple vertices of $U$, there
are at least $\left|U_{v}\right|$ uncertain edges containing both
$v$ and a vertex $u\in U$. Let $A_{1}$ be the assignment of all
vertices in $U$ to part 1 and $A_{2}$ the assignment
of all such vertices to part 2. Recalling \cref{eq:undetermined-irrelevant}, we will now give a lower bound
for
\begin{align*}
 & \E\left[Z_{v}\cond A_{2},\omega\left(v\right)=1\right]-\E\left[Z_{v}\cond A_{1},\omega\left(v\right)=1\right]\\
 & \quad=\left(\E\left[Z_{v}\cond A_{2},\omega\left(v\right)=1\right]-\E\left[Z_{v}\cond A_{2}\right]\right)-\left(\E\left[Z_{v}\cond A_{1},\omega\left(v\right)=1\right]-\E\left[Z_{v}\cond A_{1}\right]\right).
\end{align*}
To this end, for each edge $e$, define
\[
P_{e}=\left(\Pr\left[E_{e}\cond A_{2},\omega\left(v\right)=1\right]-\Pr\left[E_{e}\cond A_{2}\right]\right)-\left(\Pr\left[E_{e}\cond A_{1},\omega\left(v\right)=1\right]-\Pr\left[E_{e}\cond A_{1}\right]\right).
\]
We consider all possible cases for $e$, as follows:
\begin{itemize}
\item if $e$ is not uncertain, then $P_{e}=0-0=0$;
\item if $e$ does not contain $v$, then $P_{e}=0-0=0$;
\item if $e$ does not contain a vertex of $U$, then 
\[
P_{e}=\left(\Pr\left[E_{e}\cond\omega\left(v\right)=1\right]-\Pr\left[E_{e}\right]\right)-\left(\Pr\left[E_{e}\cond\omega\left(v\right)=1\right]-\Pr\left[E_{e}\right]\right)=0;
\]
\item if $e$ contains $v$ and a vertex $u\in U$, then:
\begin{itemize}
\item if $e$ contains no determined vertices, then $P_{e}\ge2^{-\left(k-1\right)}-\left(-2^{-\left(k-1\right)}\right)\ge2^{-\left(k-2\right)}$;
\item if $e$ contains a determined vertex in part 1, but no determined
vertex in part 2, then $P_{e}\ge0-\left(-2^{-\left(k-2\right)}\right)\ge2^{-\left(k-2\right)}$;
\item if $e$ contains a determined vertex in part 2, but no determined
vertex in part 1, then $P_{e}\ge2^{-\left(k-2\right)}-0\ge2^{-\left(k-2\right)}$.
\end{itemize}
\end{itemize}
We have shown that $P_{e}\ge0$ and $P_{e}\ge2^{-\left(k-2\right)}$
if $e$ is uncertain and contains both $v$ and a vertex $u\in U$.
Therefore,
\[
\E\left[Z_{v}\cond A_2,\omega\left(v\right)=1\right]-\E\left[Z_{v}\cond A_1,\omega\left(v\right)=1\right]\ge\left|U_{v}\right|/2^{k-2}.
\]
This means that there is a choice $A\in\left\{ A_{1},A_{2}\right\} $,
such that $\E\left[Z_{v}\cond A,\omega\left(v\right)=1\right]$ differs
by at least $\left|U_{v}\right|/2^{k-1}$ from $\E Z_{v}=\E\left[Z_{v}\cond A\right]$.
This choice of $A$, and some choice of a part for $v$, proves (1).
\end{proof}
Now, at the end of the procedure, let $D$ be the set of vertices
which did not get marked as undetermined during their step. Then the
total number of determined vertices at the end of the procedure is
$\left|D\right|+\sum_{v}\left|U_{v}\right|$, because each vertex
in $U_{v}$ becomes determined at step $v$ by part (1) of the above
claim. Note that each vertex in $W$, being among the first two vertices in
some edge $e$, becomes determined at the point when we reach the
third vertex in $e$, at the latest. Indeed, when we reach the third vertex $v$ of an edge $e$, either both of the first two vertices have already been determined, or else $e$ has at most one determined vertex and is therefore uncertain ($e$ could still either receive two different colours or not). In the latter case, $U_v$ contains all the edges of $e$ that were previously marked as undetermined, and therefore they become determined at this step. It follows that
\[
\left|D\right|+\sum_{v}\left|U_{v}\right|\ge\left|W\right|.
\]
Let $Z_{n+1}$ be the eventual size of the 2-cut produced by the procedure. Recalling part (2) of the above claim, note that $\E Z_{v+1}=\max_{A}\E\left[Z_{v}\cond A\right]$, where the maximum is over all assignments $A$ of the vertices in $U\cup\left\{ v\right\} $. The expected excess of our 2-cut is
\[
\E Z_{n+1}-\E Z_{1}=\sum_{v=1}^{n}(\E Z_{v+1}-\E Z_{v})=\sum_{v=1}^{n}(\max_{A}\E\left[Z_{v}\cond A\right]-\E Z_{v}).
\]
Now, for each $v\in D$, we have $\max_{A}\E\left[Z_{v}\cond A\right]-\E Z_{v}\ne 0$, so, recalling that all conditional expectations are an integer multiple of $1/2^{k-1}$, the expected excess is at least $\left|D\right|/2^{k-1}$. By part (1) of the above claim, this expected excess is also at least $\sum_{v}\left|U_{v}\right|/2^{k-1}$.
Taking the average of these two lower bounds, the expected excess is therefore at least
\[
\frac{\left|D\right|+\sum_{v}\left|U_{v}\right|}{2^{k}}\ge\frac{\left|W\right|}{2^{k}}.
\]
The desired result follows.
\end{proof}
The 2-cut case of \cref{thm:excess-n} is a very basic corollary of
\cref{thm:excess-W}, as follows.
\begin{proof}[Proof of the 2-cut case of \cref{thm:excess-n}]
Suppose that $k > 2$ and $H$ is a $k$-multigraph with no isolated vertices. Consider
a random ordering of the vertices of $H$. For any vertex $v$, there is an edge containing $v$, and the probability that $v$ is among one of the first two vertices of that edge is $2/k$. Therefore, in the notation of \cref{thm:excess-W},
\[
\E\left|W\right|=\sum_{v}\Pr\left(v\in W\right)\ge\left(2/k\right)n.
\]
Applying \cref{thm:excess-W} with an ordering satisfying $\left|W\right|\ge2n/k$,
we conclude that there is a 2-cut with excess at least $n/\left(k2^{k-1}\right)$,
as desired.
\end{proof}
Note that with essentially the same proof (considering a random ordering),
we have a corresponding result for 2-cuts of mixed $k$-graphs.
\begin{lemma}
\label{lem:excess-mixed}For any fixed $k\ge3$, consider a
mixed $k$-multigraph $H$ such that at least $n$ vertices are contained in edges of size at least three. Then $H$ has
a 2-cut with excess at least $n/\left(k2^{k-1}\right)$.
\end{lemma}

Combining this with the fact (mentioned in the introduction) that a multigraph with $n$ non-isolated vertices has a cut with excess at least $n/6$, we have the following result, which we need below.

\begin{corollary}
\label{cor:excess-mixed}For any fixed $k\ge2$, consider a
mixed $k$-multigraph $H$ such that at least $n$ vertices are contained in edges of size $k$. Then $H$ has
a 2-cut with excess at least $n/\left(k2^k\right)$.
\end{corollary}

Now we use \cref{lem:excess-mixed} to prove \cref{thm:excess-n} in
the case where $r>2$.
\begin{proof}[Proof of the $r>2$ case of \cref{thm:excess-n}]
Let $V=V\left(H\right)$. As we have seen, an $r$-cut can be represented
as a function $\omega:V\to\left\{ 1,\dots,r\right\} $. Consider the
``information restriction'' function $P$, defined as follows. For
any $\omega:V\to\left\{ 1,\dots,r\right\} $, let $P\left(\omega\right)=P\omega$
be the function $V\to\left\{ *,3,\dots,r\right\} $ with
\begin{align*}
P\omega\left(v\right) & =\begin{cases}
* & \text{if }\omega\left(v\right)\in\left\{ 1,2\right\} ,\\
\omega\left(v\right) & \text{otherwise}.
\end{cases}
\end{align*}
For $\rho:V\to\{*,3,\dots,r\}$, let $\Hpart{\rho}$ be the mixed $(k-r+2)$-multigraph
with edge multiset
\[
\left\{ e\cap\rho^{-1}\left(*\right): e \in E(H), \rho\left(e\right)\supseteq\left\{ 3,\dots,r\right\} ,\,\left|e\cap\rho^{-1}\left(*\right)\right|\ge2\right\}
\]
and vertex set $\rho^{-1}\left(*\right)$. Note that a 2-cut
of $\Hpart{\rho}$ with size $z$ corresponds to an $r$-cut $\omega$
of $H$ with the same size $z$ satisfying $P\omega=\rho$. 

Now, let $\omega$ be a uniformly random $r$-cut of $H$ and let
$Z$ be its size. Let $\E\left[Z\cond P\omega\right]$ be the conditional expectation of $Z$ given $P\omega$, which is precisely the average size of a uniformly random 2-cut of $\Hpart{P\omega}$. Let $Q$ be the number of vertices in $\Hpart{P\omega}$ which are contained in an edge of size $k-r+2$. By \cref{cor:excess-mixed}, for any outcome of $P\omega$, the multihypergraph $\Hpart{P\omega}$ has a 2-cut with size at least $\E\left[Z\cond P\omega\right]+Q/\left((k-r+2)2^{k-r+2}\right)$.

Each vertex $v$ of $H$ is contained in at least one edge $e$ of size $k$ and the probability that $e$ corresponds to a size-$(k-r+2)$ edge in $\Hpart{P\omega}$, containing $v$, is $\Omega(1)$. For example, this occurs if $v$ and $k-r+1$ other vertices are put in part $1$, and the remaining $r-2$ vertices are put in parts $3,\dots,r$. It follows that $\E Q=\Omega(n)$ and, therefore,
$$\E\left[\E\left[Z\cond P\omega\right]+\frac{Q}{(k-r+2)2^{k-r+2}}\right]=\E Z+\Omega(n).$$
We deduce that there is an outcome of $P\omega$ such that $\Hpart{P\omega}$ has a 2-cut with size $\E Z+\Omega(n)$, which corresponds
to an $r$-cut of $H$ with the same size, and therefore with excess $\Omega\left(n\right)$.
\end{proof}

\section{\label{sec:nontrivial-reductions}Reductions for the proof of \texorpdfstring{\cref{thm:nontrivial}}{Theorem~\ref{thm:nontrivial}}}

We now turn to the proof of \cref{thm:nontrivial}. In this section we make a few simple reductions in preparation for \cref{subsec:3-cuts-almost-linear,subsec:2-cuts-almost-linear}, which will comprise the heart of the proof of \cref{thm:nontrivial}.

First, note that if $H$ is very dense, then it trivially has a large-excess $r$-cut by \cref{cor:basic-m-n}. In fact, it is not hard to show that $H$ has a large-excess $r$-cut whenever it is ``far'' from being a linear hypergraph in various senses. In \cref{subsec:almost-linear-reduction} we collect a number of lemmas of this type.

Second, recall that in the proof of \cref{thm:excess-n}, we used a partial exposure trick to deduce a general theorem about $r$-cuts in $k$-graphs from a corresponding result for 2-cuts in mixed multihypergraphs. In \cref{subsec:special-reduction} we make some reductions of this type for \cref{thm:nontrivial}, essentially showing that it suffices to consider the case of 3-cuts in 3-multigraphs and the case of 2-cuts in mixed multihypergraphs.

By the end of the section, we will have shown that in order to prove \cref{thm:nontrivial}, it suffices to prove two lemmas (\cref{lem:almost-linear-3,lem:almost-linear-2}) concerning 3-cuts in almost-linear 3-graphs and concerning 2-cuts in almost-linear mixed hypergraphs. These lemmas will be proved in \cref{subsec:3-cuts-almost-linear,subsec:2-cuts-almost-linear}. 

\subsection{\label{subsec:almost-linear-reduction}Reduction to the almost-linear case}

In a linear hypergraph, all degrees are $O\left(n\right)$ and
the joint degree (or 2-degree) of any pair of vertices is at most
1. We first show how to reduce to the case where, except for a small
subset of bad vertices, the degrees and 2-degrees are not too
large, that is, $H$ is almost linear. We use notation of the form $\deg_H(u)$ and $\deg_H(u,v)$ to emphasise that degrees and 2-degrees are relative to $H$ and not to any induced subgraph $H[U]$.

\begin{lemma}
\label{lem:2-degree-structure}For any fixed $2\le r\le k$, consider
an $n$-vertex, $m$-edge $k$-multigraph $H$. For any $q,g,\Delta$, at least one of the following holds: there is an $r$-cut with excess $\Omega\left(gq\right)$ or there
is a vertex subset $U$ with $\left|U\right|\ge n-2q-km/\Delta$ such
that for distinct $u,v\in U$, $\deg_H\left(v\right)\le\Delta$
and $\deg_H\left(u,v\right)\le g$.
\end{lemma}

For the proof of \cref{lem:2-degree-structure} we will need the simple observation that, for a uniformly random $r$-cut in a multihypergraph, conditioning on the event that some disjoint edges are ``partially multicoloured'' increases the probability that an edge containing them is multicoloured.

\begin{lemma}
\label{lem:multi-prob-increase}
Consider a uniformly random $r$-cut of a multihypergraph $H$ and, for any $f\subseteq V(H)$ and $\ell\le r$, let $E_{f,\ell}$ be the event that $f$ has vertices in at least $\ell$ different parts (so that $E_f:=E_{f,r}$ is the event that $f$ is multicoloured). Then, for any $e\in E(H)$ of size $k$, any disjoint $f_1,\dots,f_s\subseteq e$ of size at least 2, and any $\ell_1,\dots,\ell_s\ge 2$,
$$\Pr\left(E_e\cond \bigcap_{i=1}^s E_{f_i,\ell_i}\right)=\frac{S\left(k,r\right)r!}{r^{k}}+\Omega\left(1\right).$$
\end{lemma}

\begin{proof}
We prove this by induction on $s$, assuming that it is true for all smaller values of $s$ (so if $s=1$ we are making no assumption). Let $E'=\bigcap_{i=1}^{s-1} E_{f_i,\ell_i}$ (if $s=1$ this event is the entire probability space). Either trivially or by the inductive hypothesis, we have $\Pr\left(E_e\cond E'\right)\ge S\left(k,r\right)r!/r^{k}$.

Given this observation, it suffices to show that $\Pr\left(E_e\cond E'\cap E_{f_s,\ell_s}\right)>\Pr\left(E_e\cond E'\cap \overline{E_{f_s,\ell_s}}\right)$, because each individual colouring has probability at least $1/r^k=\Omega(1)$. This is equivalent to showing that $\Pr\left(E_e\cond E_{f_s,\ell_s}\right)>\Pr\left(E_e\cond \overline{E_{f_s,\ell_s}}\right)$ in the \emph{conditional} probability space where we are conditioning on $E'$. So, for the rest of the proof, we work in this conditional probability space. Let $f=f_s$ and $\ell=\ell_s$.

For each particular choice $\rho:f\to \{1,\dots,r\}$ of the parts of vertices in $f$, let $A_\rho$ be the event that our random cut agrees with $\rho$ on $f$. Then $E_{f,\ell}$ is the disjoint union of all $A_\rho$ with $|\rho(f)|\ge \ell$ and $\overline{E_{f,\ell}}$ is the disjoint union of all $A_{\rho'}$ with $|\rho'(f)|< \ell$. Now, consider the individual conditional probabilities $\Pr\left(E_e\cond A_\rho\right)$. If $\rho'(f)\subsetneq\rho(f)$, then a simple coupling argument shows that $\Pr\left(E_e\cond A_{\rho'}\right)<\Pr\left(E_e\cond A_\rho\right)$. But, by symmetry, the probabilities $\Pr\left(E_e\cond A_\rho\right)$ only depend on $|\rho(f)|$.

For every $\rho$ with $|\rho(f)|\ge \ell$ and every $\rho'$ with $|\rho'(f)|< \ell$, we have $|\rho'(f)|<|\rho(f)|$, so by the above discussion $\Pr\left(E_e\cond A_{\rho'}\right)<\Pr\left(E_e\cond A_\rho\right)$. Observe that $\Pr\left(E_e\cond E_{f,\ell}\right)$ is an average of all the $\Pr\left(E_e\cond A_{\rho}\right)$ with $|\rho(f)|\ge \ell$ and, similarly, $\Pr\left(E_e\cond \overline{E_{f,\ell}}\right)$ is an average of all the $\Pr\left(E_e\cond A_{\rho'}\right)$ with $|\rho'(f)|< \ell$, so the desired conclusion follows.
\end{proof}

Now we prove \cref{lem:2-degree-structure}.

\begin{proof}[Proof of \cref{lem:2-degree-structure}]
Let $G_{g}\left(H\right)$ be the graph of pairs of vertices $u,v$
such that $\deg_H\left(u,v\right)>g$. If $G_{g}\left(H\right)$ has
an independent set of size $n-2q$, then we can remove the
(at most $km/\Delta$) vertices with degree greater than $\Delta$
and we are done. Otherwise, $G_{g}\left(H\right)$ has a matching
$M$ of size $q$. Now, consider a uniformly random $r$-cut and
condition on the event that each edge of the matching has its vertices
in two different parts. By \cref{lem:multi-prob-increase}, due to the conditioning, each of the $\Omega\left(gq\right)$
edges of $H$ which contains an edge of $M$ is multicoloured with
probability 
\[
\frac{S\left(k,r\right)r!}{r^{k}}+\Omega\left(1\right)
\]
and the probability that every other edge is multicoloured is unaffected
by the conditioning. Hence, the expected excess of this random $r$-cut is
$\Omega\left(gq\right)$ and so there is a particular $r$-cut with this excess.
\end{proof}

In addition to having few bad vertices, we
would also like to know that a good proportion
of the edges do not involve the bad vertices. For 2-cuts,
this is easy to obtain.

\begin{lemma}
\label{lem:almost-linear-reduction-2}Fix $k\ge2$ and consider an
$n$-vertex, $m$-edge $k$-multigraph $H$. Suppose $U$ is a 
vertex subset with $\left|U\right|=n-b$. Then, at least one of the following holds: there
is a 2-cut with excess $\Omega\left(m/b\right)$ or $e\left(H\left[U\right]\right)=\Omega\left(m\right)$.
\end{lemma}

\begin{proof}
Let $V=V\left(H\right)$. If there are $\left(1-c\right)m$ edges
intersecting both $U$ and $V\setminus U$, for sufficiently small
$c>0$, then the 2-cut with parts $U$ and $V\setminus U$ itself has
excess $\Omega\left(m\right)$. Otherwise, at least one of $U$ or $V\setminus U$
induces $\Omega\left(m\right)$ edges. If $V\setminus U$ induces
$\Omega\left(m\right)$ edges, then $e\left(G\left(H\right)\left[V\setminus U\right]\right)=\Omega\left(m\right)$
as well, so, by \cref{lem:advanced-m-n}, there is a 2-cut with excess
$\Omega\left(m/b\right)$.
\end{proof}

Next, we observe that the case of 3-cuts in 3-graphs
reduces to the case of 2-cuts in 3-graphs.

\begin{lemma}
\label{thm:2-cut-to-3-cut}If a 3-multigraph has a 2-cut with excess
$x$, then it has a 3-cut with excess $\Omega\left(x\right)$.
\end{lemma}

\begin{proof}
Note that $S\left(3,2\right)2!/2^{3}=3/4$ and $S\left(3,3\right)3!/3^{3}=6/27$. Given a 2-cut of size $\left(3/4\right)m+ x$,
create a third part by including each vertex independently with probability
$1/3$. The probability an edge that spans the 2-cut is multicoloured in the 3-cut is $2\left(1/3\right)\left(2/3\right)^{2}=8/27$,
so the expected size of our random 3-cut is $(6/27)m+\Omega(x)$.
\end{proof}

As a consequence, the next lemma immediately follows from \cref{lem:almost-linear-reduction-2}.

\begin{lemma} \label{lem:almost-linear-reduction-3}
Consider an $n$-vertex, $m$-edge 3-multigraph $H$. Suppose $U$ is a 
vertex subset with $\left|U\right|=n-b$. Then, at least one of the following holds: there
is a 3-cut with excess $\Omega\left(m/b\right)$ or $e\left(H\left[U\right]\right)=\Omega\left(m\right)$.
\end{lemma}

For general $r$-cuts in $k$-graphs, it is not clear how to reduce
to the case where $H\left[U\right]$ has many edges, but we can reduce
instead to the case where there are many edges of $H$ which have all but
at most one of their vertices in $U$.

\begin{lemma}
\label{lem:almost-linear-reduction-k}Fix $2\le r\le k$ and consider
an $n$-vertex, $m$-edge $k$-multigraph $H$. Suppose $U$ is a 
vertex subset with $\left|U\right|=n-b$. Then, at least one of the following holds: there
is an $r$-cut with excess $\Omega\left(m/b\right)$ or there are
$\Omega\left(m\right)$ edges which have at least $k-1$ of their
vertices in $U$.
\end{lemma}

\begin{proof}
Let $V = V(H)$. If $e\left(G\left(H\right)[V \setminus U]\right)=\Omega\left(m\right)$,
then, by \cref{lem:advanced-m-n}, there is an $r$-cut with excess $\Omega\left(m/b\right)$.
Otherwise, almost all of the $m$ edges have at most one of their
vertices in $V\setminus U$.
\end{proof}

To summarise, it only remains to treat the case where $H$ contains a large
subset $U$ of vertices with the property that many edges are almost completely 
contained in $U$ and the degrees and 2-degrees in $U$ are not too high.

\subsection{Reducing to some special values of $k$ and $r$\label{subsec:special-reduction}}

Recall that in the proof of \cref{thm:excess-n}, we were able to reduce a general $r$-cut problem to a 2-cut problem by making a choice for the vertices in parts $3,\dots,r$ and then passing to an auxiliary mixed $\left(k-r+2\right)$-multigraph $\Hpart{}$. We will use the same idea here, but since 2-graphs and 3-graphs might only have small 2-cuts, it will only be useful when $r\le k-2$. To be more specific, we will use this idea to deduce the following lemma from another result.

\begin{lemma}
\label{lem:almost-linear-k}Fix $2\le r\le k-2$ and consider an $m$-edge,
$n$-vertex $k$-multigraph $H$ with a distinguished vertex subset
$U$. Suppose that $H$ has $\Omega\left(m\right)$ edges which have
at least four of their vertices in $U$,
$\deg_H\left(u\right)\le\Delta$ for all $u\in U$, and $\deg_H\left(u,v\right)\le g$
for all distinct $u,v\in U$. Let $p=\Delta^{-3/5}\land\left(g^{-2/3}\Delta^{-1/3}\right)$
and suppose $g=o(pm)$ and $g\log n=o(p\Delta)$.
Then $H$ has an $r$-cut with excess $\Omega\left(\sqrt{p}m/\sqrt{\Delta}\right)$.
\end{lemma}

\cref{lem:almost-linear-k} basically says that almost-linear $k$-graphs have large-excess $r$-cuts. Combined with the results of \cref{subsec:almost-linear-reduction}, this will suffice to prove the case $r\le k-2$ of \cref{thm:nontrivial}. The details will be given at the end of this subsection. As promised, we will now use the $\Hpart{}$ trick to show that \cref{lem:almost-linear-k} is a consequence of the following 2-cut lemma.

\begin{lemma}
\label{lem:almost-linear-2}Fix $k\ge4$ and consider an $m$-edge,
$n$-vertex mixed $k$-multigraph $H$ with a distinguished vertex subset
$U$. Suppose that $H\left[U\right]$ has $\Omega\left(m\right)$
edges of size at least four, $\deg_H\left(u\right)\le\Delta$ for all $u\in U$, and
$\deg_H\left(u,v\right)\le g$ for all distinct $u,v\in U$. Let $p=\Delta^{-3/5}\land\left(g^{-2/3}\Delta^{-1/3}\right)$
and suppose $g=o(pm)$ and $g\log n=o(p\Delta)$.
Then $H$ has a 2-cut with excess $\Omega\left(\sqrt{p}m/\sqrt{\Delta}\right)$.
\end{lemma}

\begin{proof}[Proof of \cref{lem:almost-linear-k} given \cref{lem:almost-linear-2}]
We will use
almost exactly the same approach as in the proof of \cref{thm:excess-n}. Recall the definition of the operator $P$ and the mixed $\left(k-r+2\right)$-multigraphs $\Hpart{\rho}$. (Informally, we remind the reader that $P$ restricts the information of whether a vertex is in part 1 or 2: if $\omega(v)\in\{1,2\}$ then $P\omega(v)=*$. The multihypergraph $\Hpart{\rho}$ on the vertex set $\rho^{-1}(*)$ is defined such that a 2-cut of $\Hpart{\rho}$ corresponds to an $r$-cut of $H$ with the same size.)

Let $\omega$ be a uniformly random $r$-cut of $H$ and let $Z$ be its size. Let $Q$ be the number of edges in $\Hpart{P\omega}[U]$ of size at least four. By \cref{lem:almost-linear-2}, there is $c=\Omega(\sqrt{p/\Delta})$ such that for each outcome of $P\omega$, the multihypergraph $\Hpart{P\omega}$ has a 2-cut of size at least $\E\left[Z\cond P\omega\right]+cQ$ (and therefore $H$ has an $r$-cut of this size).

Now, the probability that an edge in $H$ with at least four vertices in $U$ corresponds to a size-4 edge in $\Hpart{P\omega}[U]$ is $\Omega\left(1\right)$, so $\E Q=\Omega\left(m\right)$ and $\Hpart{P\omega}$ has a 2-cut of size at least
$$\E \left[\E\left[Z\cond P\omega\right]+cQ\right]=\E Z+\Omega(cm).$$
This corresponds
to an $r$-cut of $H$ which has the same size and, therefore, has excess $\Omega\left(cm\right)=\Omega\left(\sqrt{p}m/\sqrt{\Delta}\right)$.
\end{proof}

We next use a slight variation of the $\Hpart{}$ trick to directly show that the $r\ge k-1$ case of \cref{thm:nontrivial} follows from the $r=k=3$ case.

\begin{proof}[Proof of the $r\ge k-1$ case of \cref{thm:nontrivial} given the $r=k=3$ case]
First, consider the case where $k=r\ge 3$. In this case we redefine $P$ and $\Hpart{\rho}$ as follows. For
any $\omega:V\to\left\{ 1,\dots,k\right\} $, let $P\left(\omega\right)=P\omega$
be the function $V\to\left\{ *,4,\dots,k\right\} $ with
\begin{align*}
P\omega\left(v\right) & =\begin{cases}
* & \text{if }\omega\left(v\right)\in\left\{ 1,2,3\right\} ,\\
\omega(v) & \text{otherwise}.
\end{cases}
\end{align*}
For $\rho:V\to\left\{ *,4,\dots,k\right\}$, let $\Hpart{\rho}$ be the 3-multigraph
with edge multiset
\[
\left\{ e\cap\rho^{-1}\left(*\right): e \in E(H), \rho\left(e\right)\supseteq\left\{ 4,\dots,k\right\} ,\,\left|e\cap\rho^{-1}\left(*\right)\right|=3\right\}
\]
and vertex set $\rho^{-1}\left(*\right)$. Note that a 3-cut
of $\Hpart{\rho}$ with size $z$ corresponds to an $r$-cut $\omega$
of $H$ with the same size $z$ satisfying $P\omega=\rho$.

We now proceed in essentially the same way as in the proofs of \cref{thm:excess-n,lem:almost-linear-k}. Let $\omega$ be a uniformly random $k$-cut of $H$, let $Z$ be its size, and let $Q=e(\Hpart{P\omega})$. By assumption, there is $c=\Omega(1)$ such that $\Hpart{P\omega}$ has a 3-cut with size $\E\left[Z\cond P\omega\right]+cQ^{5/9}$.

The probability that an edge in $H$ corresponds to an edge in $\Hpart{P\omega}$ is $(k!/k^k)/(3!/3^3)=\Omega\left(1\right)$, so $\E Q=\Omega\left(m\right)$ and, applying Markov's inequality to $m-Q$, we have $$\Pr\left(Q<\E Q/2\right)=\Pr\left(m-Q>(1+\Omega(1))\E [m-Q]\right)=1-\Omega\left(1\right).$$
This implies that $\E Q^{5/9}\ge\Pr\left(Q\ge\E Q/2\right)(\E Q/2)^{5/9}=\Omega(m^{5/9})$.
So, $\Hpart{P\omega}$ has a 2-cut of size at least
$$\E \left[\E\left[Z\cond P\omega\right]+cQ^{5/9}\right]=\E Z+\Omega(m^{5/9}).$$
This corresponds
to an $r$-cut of $H$ with the same size and, therefore, with excess $\Omega\left(m^{5/9}\right)$, as desired.

Finally, it remains to consider the case where $r=k-1$, for $k\ge4$.
In this case we observe that our $k$-uniform $r$-cut problem corresponds
to an $r$-multigraph $r$-cut problem in much the same way that 3-graph 2-cut problems correspond to multigraph cut problems (recall that this was important for \cref{thm:3-2}). Let $H'$ be the $km$-edge $r$-multigraph obtained from $H$
by taking each $k$-edge and replacing it with an $r$-graph $k$-clique
on its vertex set. Now, there is essentially only one way that an edge of $H$ can be multicoloured (with a single repeated part), so an $r$-cut of $H'$ with size $z$ corresponds
precisely to an $r$-cut of $H$ with size $z/2$. Since we have already
proved that $\Omega\left(m\right)$-edge $r$-multigraphs have $r$-cuts
with excess $\Omega\left(m^{5/9}\right)$, the desired result follows.
\end{proof}

To prove the $r=k=3$ case of \cref{thm:nontrivial}, we will combine the lemmas in \cref{subsec:almost-linear-reduction} with the following result.

\begin{lemma}
\label{lem:almost-linear-3}Consider an $m$-edge, $n$-vertex 3-multigraph
$H$ with a distinguished vertex subset $U$. Suppose that $H\left[U\right]$
has $\Omega\left(m\right)$ edges, $\deg_H\left(u\right)\le\Delta$
for all $u\in U$, and $\deg_H\left(u,v\right)\le g$ for all distinct $u,v\in U$.
Let $p=\Delta^{-3/5}\land\left(g^{-2/3}\Delta^{-1/3}\right)$ and
suppose $g=o(pm)$ and $g\log n=o(p\Delta)$. Then $H$ has a 3-cut
with excess $\Omega\left(\sqrt{p}m/\sqrt{\Delta}\right)$.
\end{lemma}

We now show how to deduce the $r=k=3$ case and the $r\le k-2$ case of \cref{thm:nontrivial} from \cref{lem:almost-linear-3,lem:almost-linear-k} and the lemmas in \cref{subsec:almost-linear-reduction}.

\begin{proof}[Proof of \cref{thm:nontrivial} given \cref{lem:almost-linear-3,lem:almost-linear-k}]

We can assume $H$ has no isolated vertices (otherwise we could delete them without changing the max-cut or the number of edges). This means $n=O(m)$. Define $\Delta=m^{5/9}$, $g=m^{3/5}\Delta^{-4/5}=m^{7/45}$
and $q=\sqrt{m/g}=m^{19/45}$.

Note that $m^{5/9}=o(gq)$, so by \cref{lem:2-degree-structure}
it suffices to consider the case where there is a set $U$ with $\left|U\right|\ge n-O\left(q+m/\Delta\right)$
such that, for all distinct $u,v\in U$, $\deg\left(v\right)\le\Delta$ and $\deg\left(u,v\right)\le g$.

If $r=k=3$, we may now apply \cref{lem:almost-linear-reduction-3}.
If $H$ has a 3-cut with excess $\Omega\left(m/\left(q+m/\Delta\right)\right)=\Omega\left(\left(m/q\right)\land\Delta\right)=\Omega\left(m^{5/9}\right)$,
then we are done. Otherwise, $e\left(H\left[U\right]\right)=\Omega\left(m\right)$.
In this case, let $p=\Delta^{-3/5}\land\left(g^{-2/3}\Delta^{-1/3}\right)=m^{-1/3}$. Note that $g= m^{7/45} =o( pm)$ and
$p\Delta=m^{2/9}$, so $g\log n=o(p\Delta)$.
We may therefore apply \cref{lem:almost-linear-3},
from which it follows that $H$ has a 3-cut with excess $\Omega\left(\sqrt{p}m/\sqrt{\Delta}\right)=\Omega\left(m^{5/9}\right)$.

If $r\le k-2$, we may proceed in basically the same way. If $r=2$ (therefore $k\ge 4$), apply \cref{lem:almost-linear-reduction-2}: if $H$ has a 3-cut with excess $\Omega\left(m/\left(q+m/\Delta\right)\right)$, we are done, so we may assume $e\left(H\left[U\right]\right)=\Omega\left(m\right)$. Otherwise, if $r\ge3$ (therefore $k\ge 5$), apply \cref{lem:almost-linear-reduction-k}: if $H$ has an $r$-cut with excess $\Omega\left(m/\left(q+m/\Delta\right)\right)$, we are done, so we may assume $H$ has $\Omega(m)$ edges with at least $k-1\ge 4$ of their vertices in $U$. In either case, we may then apply \cref{lem:almost-linear-k} with $p=\Delta^{-3/5}\land\left(g^{-2/3}\Delta^{-1/3}\right)=m^{-1/3}$ to prove that $H$ has an $r$-cut with excess $\Omega\left(\sqrt{p}m/\sqrt{\Delta}\right)=\Omega\left(m^{5/9}\right)$.
\end{proof}

It remains to prove \cref{lem:almost-linear-3,lem:almost-linear-2}.

\section{3-cuts in almost-linear 3-multigraphs\label{subsec:3-cuts-almost-linear}}

In this section, we prove \cref{lem:almost-linear-3}. The idea is to consider a uniformly random 3-cut but to expose only the
information about whether each vertex is in part 3 or not, meaning that
for vertices which are not chosen to be in part 3, we are free
to choose whether they should be in part 1 or part 2. (This is exactly
as in the proofs of \cref{thm:excess-n,lem:almost-linear-k}.) We thereby reduce the problem
of finding a large-excess 3-cut in $H$ to the problem of finding
a large 2-cut in the random multigraph $\Gpart{}$ obtained by considering
each hyperedge of $H$ which has exactly one vertex in part 3, deleting
that vertex, and adding the resulting pair of vertices as an edge
in $\Gpart{}$. 

Now, we expect that in a random multigraph we can find a large-excess 2-cut
in a greedy fashion: we simply go through the vertices one-by-one
and place each vertex in the part that maximizes the number of multicoloured
edges between that vertex and the preceding vertices. Doing this will
result in at least half the edges being multicoloured, with equality
only if at each step both choices of part for the current
vertex are equally good. Since we have some variation due to randomness,
it's unlikely that both choices will be equally good at most steps.

We cannot immediately use this method to find a large-excess cut of
the random multigraph $\Gpart{}$, because the edges of $\Gpart{}$ are
not completely independent. In particular, if two edges of $H$ intersect,
then the edges of $\Gpart{}$ that can arise from $H$ depend on each
other. But since we are assuming $H$ is almost linear, there are
not too many of these dependencies, so if we look
at the subgraph of $\Gpart{}$ induced by a small random set of vertices,
we expect most of the edges in this induced subgraph to be independent.
Therefore, we can find a large-excess 2-cut in many small induced subgraphs
of $\Gpart{}$ and then randomly combine these 2-cuts to obtain a large-excess
2-cut of $\Gpart{}$ itself.

In preparation for proving \cref{lem:almost-linear-3}, we need
a few lemmas. First, we will need to be able to split our
vertex set into many small subsets such that the subgraph
of $\Gpart{}$ induced by each of these vertex sets has many edges
but few dependencies.

\begin{definition}
\label{def:good-3}Consider a 3-multigraph $H$ with a distinguished vertex subset $U$. Consider a partition $\cV=\left(V_{1},\dots,V_{t}\right)$
of $U$. We say that $\cV$ is \emph{$\left(m',\Delta'\right)$-good}
with respect to $\left(H,U\right)$ if: 
\begin{enumerate}
\item[(i)] $\bigcup_{i}G\left(H[U]\right)\left[V_{i}\right]$ has
at least $m'$ edges; 
\item[(ii)] $\bigcup_{i}G\left(H\right)\left[V_{i}\right]$ has maximum degree
at most $\Delta'$; 
\item[(iii)] no edge of $H$ is completely contained in any $V_i$;
\item[(iv)] for each $i$, any two edges $e_1,e_2\in E(H)$ that both intersect $V_i$ in two vertices must satisfy $e_1\cap e_2\cap U\setminus V_i=\emptyset$, unless they are coincident edges on the same set of 3 vertices.
\end{enumerate}
We say that $\cV$ is \emph{$y$-almost }$\left(m',\Delta'\right)$-good
with respect to $(H,U)$ if properties (i) and (ii) hold, and if at most $y/2$ edges
of $H$ fail to satisfy property (iii) and at most $y/2$ pairs of
edges of $H$ fail to satisfy property (iv).
\end{definition}

To explain property (iv), the idea is that if $e_1,e_2\in E(H)$ both intersect $V_i$ in two vertices, then the third vertex of $e_1$ is different to the third vertex in $e_2$, except for complications caused by multiple coincident edges and complications caused by the few high-degree vertices outside $U$.

Next, the following lemma allows us to find an almost-good partition in the setting of \cref{lem:almost-linear-3} where degrees and codegrees are controlled.

\begin{lemma}
\label{lem:exists-good-3}Consider an $m$-edge, $n$-vertex $3$-multigraph
$H$ with a distinguished vertex subset $U$. Suppose that $H\left[U\right]$
has $\Omega\left(m\right)$ edges, $\deg_H\left(u\right)\le\Delta$
for all $u\in U$, and $\deg_H\left(u,v\right)\le g$ for all distinct $u,v\in U$.
Let $p=\Delta^{-3/5}\land\left(g^{-2/3}\Delta^{-1/3}\right)$ and
suppose $g=o( pm)$ and $g\log n=o(p\Delta)$. Then, for any $c>0$,
there are $m'=\Omega\left(pm\right)$ and $\Delta'=O\left(p\Delta\right)$
such that there exists a $\left(cm'/\sqrt{\Delta'}\right)$-almost
$\left(m',\Delta'\right)$-good partition $\cV=\left(V_{1},\dots,V_{t}\right)$
with respect to $\left(H,U\right)$.
\end{lemma}

To prove \cref{lem:exists-good-3}, we may simply take a random partition and show that it satisfies the desired almost-goodness property with positive probability. We defer the details of the proof to \cref{subsec:2-cuts-almost-linear}, where we will prove a generalisation, \cref{lem:exists-good}, of \cref{lem:exists-good-3}.

The next lemma allows us to combine large-excess cuts on many vertex
subsets into a single large-excess cut. 

\begin{lemma} \label{lem:combine-cuts}Consider
a graph $G$ with $m$ edges and suppose there are disjoint vertex
subsets $V_{1},\dots,V_{t}$, with $V_i$ inducing a graph $G_{i}$ with a
2-cut of excess $x_{i}$. Then $G$ has a 2-cut with excess at least $\sum_{i}x_{i}$.
\end{lemma}

\begin{proof} We can assume $V_{1},\dots,V_{t}$ form a
partition of $V\left(G\right)$, adding singleton parts if necessary.
Consider a cut $\omega_{i}:V_{i}\to\left\{ 1,2\right\}$ of each $G_i$ with excess
$x_{i}$. For each $i$, let $\sigma_{i}:\left\{ 1,2\right\} \to\left\{ 1,2\right\} $
be a uniformly random permutation and define the random cut $\omega:V\left(G\right)\to\left\{ 1,2\right\} $
by $\omega\left(v\right)=\sigma_{i}\left(\omega_{i}\left(v\right)\right)$
for $v\in V_{i}$. Among the $m'=e\left(\bigcup_{i}G\left[V_{i}\right]\right)$
edges within the parts, $\sum_{i}\left(e\left(G\left[V_{i}\right]\right)/2+x_{i}\right)=m'/2+\sum_{i}x_{i}$
of them are multicoloured, while the probability every other edge is
multicoloured is $1/2$. Therefore, the expected excess of $\omega$
is $\sum_{i}x_{i}$, so there must be an outcome of $\omega$ with
at least this excess. \end{proof}

Next, we need the fact that random graphs have large-excess cuts.
This is a corollary of a more general result, \cref{lem:random-cut},
which we will prove in \cref{subsec:2-cuts-almost-linear}.
\begin{lemma}
\label{lem:random-subgraph-cut}Fix constant $p\in (0,1)$ and let $G$ and $F$ be multigraphs on the same set of $n$ vertices such that $G$ has $m$ edges and maximum degree $\Delta$. Suppose, moreover, that the edges of $G$ are partitioned into groups, each only containing edges between the same two vertices. Let $R$ be the random multigraph obtained by including each group independently with probability $p$ and let $X^{\max}$ be the maximum excess of a 2-cut of $F\cup R$. Then $\E X^{\max}=\Omega\left(m/\sqrt{\Delta}\right)$.
\end{lemma}

Now we can prove \cref{lem:almost-linear-3}.
\begin{proof}[Proof of \cref{lem:almost-linear-3}]
Let $V=V\left(H\right)$. Recall the operator $P$ and the multihypergraphs
$\Hpart{\rho}$ defined in the proof of \cref{thm:excess-n}. (Informally, we remind the reader that $P$ restricts the information of whether a vertex is in part 1 or 2: if $\omega(v)\in\{1,2\}$, then $P\omega(v)=*$.) In the case of 3-cuts in 3-graphs $\Hpart{\rho}$ is always a multigraph, so we write $\Gpart{\rho}=\Hpart{\rho}$. As outlined at the beginning of this section, $\Gpart{\rho}$ may be obtained by considering each hyperedge of $H$ which has exactly one vertex in part 3, deleting
that vertex, and adding the resulting pair of vertices as an edge
in $\Gpart{\rho}$. Recall that a 2-cut
of $\Gpart{\rho}$ with size $z$ corresponds to a 3-cut $\omega$
of $H$ with the same size $z$ satisfying $P\omega=\rho$.

Consider some small $c>0$ and apply
\cref{lem:exists-good-3} to find a $cx$-almost $\left(2m',\Delta'\right)$-good
partition $\cV=\left(V_{1},\dots,V_{t}\right)$ with respect to $(H,U)$,
where $m'=\Omega\left(pm\right)$, $\Delta'=O\left(p \Delta\right)$ and $x=m'/\sqrt{\Delta'}=\Omega(\sqrt p m/\sqrt \Delta)$.
Delete at most $cx$ edges of $H$ (causing corresponding changes in the $\Gpart{\rho}$) so that $\cV$ becomes an $\left(m',\Delta'\right)$-good
partition with respect to $(H,U)$. We will show that the resulting 3-graph $H$
has a 3-cut with excess $\Omega\left(x\right)$, where the implied constant does not depend on $c$. This suffices to prove the lemma, because, for sufficiently
small $c>0$, restoring the deleted edges will not have too significant
an impact on the excess of the 3-cut.

Let $\omega$ be a uniformly random 3-cut of $H$,
so that $\Gpart{P\omega}$ is a random multigraph. For each $i$ let $V_i' = V_i\cap V(\Gpart{P\omega})=\left\{ v\in V_{i}:\omega\left(v\right)\in\left\{ 1,2\right\} \right\}$. The point of the good partition $\cV$ is that while $\Gpart{P\omega}$ has dependencies between its edges, if we condition on some $V_i'$ and individually look at the induced subgraph $\Gpart{P\omega}[V_i']$, it is a random subgraph of $G(H)[V_i']$ with essentially independent edges. This is because for every edge $e\in E(G(H)[V_i'])$ (arising from the edge $e\cup\{v_e\}\in E(H)$, say), the presence of $e$ in $\Gpart{P\omega}[V_i']$ solely depends on the random value $P\omega(v_e)$. By property (iii) of goodness the $v_e$ are outside $V_i'$ and by property (iv), except for repetitions due to multiple coincident edges, all the $v_e$ inside $U$ are distinct. That is to say, each of the edges of $G(H[U])[V_i']$ have their own independent sources of randomness. Let $\Ggood{i}=G(H[U])[V_i']$ and let $\Gbad{i}$ contain the edges $e$ in $G(H)[V_i']$ which are not in $\Ggood{i}$, meaning that $v_e\in V\setminus U$ (these $v_e$ may not be distinct). See \cref{fig:good} for an illustration of the above considerations.

\begin{figure}[h]
\begin{center}
\includegraphics{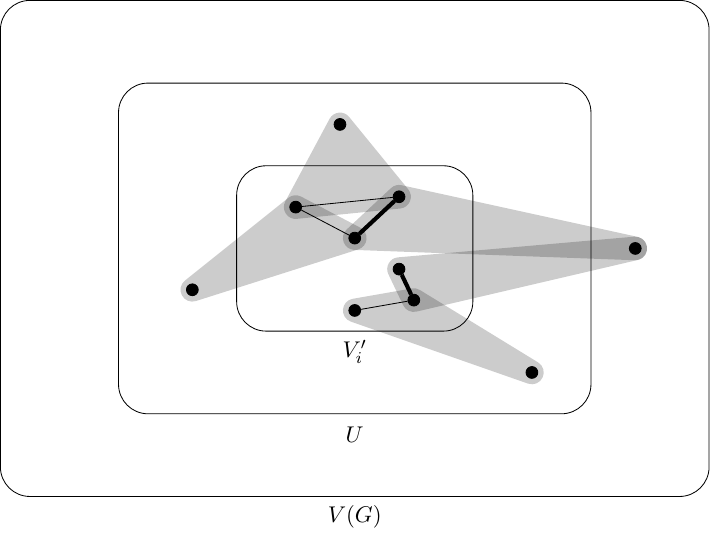}
\end{center}

\caption{\label{fig:good}A close-up picture of some $V_i'$. Each edge of $H$ intersecting $V_i'$ in two vertices contributes an edge to $G(H)[V_i']$. The thinner edges correspond to $\Ggood{i}=G(H[U])[V_i']$ and the thicker edges correspond to $\Gbad{i}$. Crucially, all edges $e\in E(\Ggood{i})$ have different $v_e$. Note that some of the edges in $G(H[U])[V_i']$ may be multiple coincident edges, but if so the corresponding edges of $H$ also coincide.}
\end{figure}

For each $i$, let $\Rgood{i}$ (respectively, $\Rbad{i}$) be the random multigraph of edges $e\in E(\Ggood{i})$ (respectively, $e\in E(\Gbad{i})$) such that $P\omega(v_e)=3$. Taking into account the above discussion, $\Gpart{P\omega}[V_i']=\Rgood{i}\cup \Rbad{i}$ and if we condition on $V_i'$, then $\Rgood{i}$ and $\Rbad{i}$ are independent. Also, let $X_{i}$ be the maximum excess of a 2-cut of $\Gpart{P\omega}\left[V_{i}'\right]$ and write $\E\left[X_i\cond V_i'\right]$ for the conditional expectation of $X_i$ given the random set $V_i'$. We will next use \cref{lem:random-subgraph-cut} to estimate $\E\left[X_i\cond V_i'\right]$. Towards this end, condition on a particular outcome of $V_i'$, so that in what follows we may consider it a fixed set of vertices. Since $\Rgood{i}$ and $\Rbad{i}$ are independent, we may also condition on any particular outcome of $\Rbad{i}$, without affecting the distribution of $\Rgood{i}$. By property (ii) of goodness, we are now in a position to apply \cref{lem:random-subgraph-cut}, with $F=\Rbad{i}$, $G=\Ggood{i}$, $R=\Rgood{i}$, $p=1/3$, $\Delta=\Delta'$ and the groups corresponding to edges of $\Ggood{i}$ coming from multiple coincident edges of $H$ on the same 3 vertices. It follows that
$$\E\left[X_i\cond V_i'\right]=\Omega\left(\frac{e(G(H[U])[V_i'])}{\sqrt{\Delta'}}\right).$$
The above is true for any outcome of $V_i'$. Now, we  return to viewing $V_i'$ as a random set (depending on $P\omega$, where $\omega$ is a uniformly random 3-cut of $H$). Each edge of $G(H[U])[V_i]$ is present in $G(H[U])[V_i']$ with probability $(2/3)^2=\Omega(1)$, so, by property (i) of goodness, with $X=\sum_{i}X_{i}$ we have
\begin{align*}
\E X&=\sum_{i}\E\left[\E\left[X_i\cond V_i'\right]\right]=\sum_{i}\Omega\left(\frac{e(G(H[U])[V_i])}{\sqrt{\Delta'}}\right)=\Omega\left(\frac{m'}{\sqrt{\Delta'}}\right)=\Omega(x).
\end{align*}
By \cref{lem:combine-cuts},
$\Gpart{P\omega}$ has a cut with size at least $e(\Gpart{P\omega})/2+X$. Each of the $3m$ edges of $G(H)$ corresponds to an edge of $\Gpart{P\omega}$ with probability $(1/3)(2/3)^2=4/27$, so $\E e(\Gpart{P\omega})=(4/9)m$ and there is an outcome of $P\omega$ such that $\Gpart{P\omega}$ (and therefore $H$) has a cut of size at least
$$\E\left[e(\Gpart{P\omega})/2+X\right]=(2/9)m+\Omega(x).$$
Observing that a random 3-cut of $H$ has expected size $(2/9)m$, we are done.
\end{proof}

\section{2-cuts in almost-linear hypergraphs\label{subsec:2-cuts-almost-linear}}

In this section, we prove \cref{lem:almost-linear-2}. Although the proof is overall more complicated than the proof of \cref{lem:almost-linear-3}, there is one simplifying reduction that we can make in this case: we do not need to worry about the ``bad'' vertices outside $U$.

\begin{lemma}
\label{lem:no-bad-reduction}
To prove \cref{lem:almost-linear-2}, it suffices to consider the case where $U=V(H)$.
\end{lemma}

To prove \cref{lem:no-bad-reduction} we will again use a variant of the $\Hpart{}$ trick, but this time there will be an extra twist. The basic idea is to consider a random 2-cut, but only expose the parts of the vertices outside $U$. One might hope that the problem of completing this partial cut to a large cut of $H$ can be directly formulated as a max-cut problem on some auxiliary multihypergraph $\Hpart{}$ with vertex set $U$. We would then hope to apply the $V(H)=U$ case of \cref{lem:almost-linear-2} and show that the resulting large-excess cut of $\Hpart{}$ corresponds, on average, to a large-excess cut of $H$. Unfortunately, fixing the parts of vertices outside $U$ introduces some asymmetry: there are now some edges $e$ which so far have vertices only in one specific part, meaning that for $e$ to be multicoloured we need to put one of the vertices of $e\cap U$ in the other part. This asymmetry cannot be expressed in a max-cut problem.

We overcome this issue with a simple averaging trick. Observe that the choices $\phi:U\to\{1,2\}$ of parts for vertices in $U$ can be paired up in a natural way: we pair $\phi$ with its ``opposite'' $\overline{\phi}:U\to\{1,2\}$ defined by $\overline{\phi}(u)=3-\phi(u)$. Now, given a choice for the parts of the vertices outside $U$, instead of trying to find $\phi:U\to\{1,2\}$ completing our 2-cut such that the size of the resulting cut of $H$ is maximised, we aim to find $\phi:U\to\{1,2\}$ such that the \emph{average} size of the two cuts completed by $\phi$ and by $\overline{\phi}$ is maximised. This problem can in fact be formulated as a 2-cut problem, as follows. For an edge $e$ intersecting $U$, which so far has vertices only in one specific part, there are two relevant possibilities for $\phi(e\cap U)$. Either $\phi(e\cap U)=\{1,2\}$, in which case $e$ is multicoloured in both of the two cuts completed by $\phi$ and by $\overline{\phi}$, or else $|\phi(e\cap U)|=1$, in which case $e$ will be multicoloured in exactly one of the two completions given by $\phi$ and $\overline{\phi}$. That is to say, the contribution of $e$ to the average size of the two cuts under consideration depends only on whether $e\cap U$ is multicoloured by $\phi$.

Of course, there are also edges $e$ of $H$ which are completely contained in $U$ and, for these edges, the averaging has no effect: either $e$ is multicoloured by both $\phi$ and $\overline \phi$ or it is multicoloured by neither. Our auxiliary multihypergraph $\Hpart{}$ will therefore consist of edges of two different types: first, the edges in $H[U]$, and second, derived edges $e\cap U$ of the type described in the previous paragraph. These two types of edges need to effectively be assigned different weights, because they contribute differently to the average size of a pair of opposite cuts. We will accomplish this by making multiple copies of the edges in $H[U]$. The details follow.

\begin{proof}[Proof of \cref{lem:no-bad-reduction}]
We will prove \cref{lem:almost-linear-2} assuming its statement holds for any multihypergraph $H$ and $U=V(H)$. Let $V=V(H)$ and let $\overline U=V\setminus U$. For a 2-cut $\omega:V\to \{1,2\}$, let $P\omega:\overline U\to \{1,2\}$ be the restriction of $\omega$ to $\overline U$.

For $\rho:\overline U \to \{1,2\}$, let $H_{\rho}'\subseteq H$ consist of
those edges $e$ in $H$ but not in $H[U]$ with $\left|\rho\left(e\setminus U\right)\right|<2$ and $|e\cap U|>0$. This means the edges $e$ in $H_{P\omega}'$ are precisely those edges that are not contained in $U$, such that given the restricted information $P\omega$ it is not yet determined whether $e$ is multicoloured by $\omega$. Let $\Hpart{\rho}$ be the mixed $k$-multigraph with vertex set $U$, containing two copies of each of the edges in $H[U]$ and also containing an edge $e\cap U$ for each $e\in E(H_{\rho}')$. Let $N^\multi_{\rho}$ be the number of edges of $H$ which can already be seen to be multicoloured by $\omega$ given the information $P\omega=\rho$.

Now, consider any $\rho:\overline U\to\{1,2\}$ and $\phi:U\to\left\{ 1,2\right\}$. Let $z^U$ be the size of $\phi$ as a 2-cut of $H[U]$, let $z'$ be the size of $\phi$ as a 2-cut of $H_{\rho}'$, and let $\zpart=2z^U+z'$ be the size of $\phi$ as a 2-cut of $\Hpart{\rho}$. Let $\omega:V\to \{1,2\}$ be the $2$-cut which is equal to $\rho$ outside $U$ and equal to $\phi$ inside $U$, and let $\overline\omega:V\to \{1,2\}$ be the $2$-cut which is equal to $\rho$ outside $U$ and equal to $3-\phi$ inside $U$.

For each edge $e\in E(H_\rho')$ whose corresponding edge in $\Hpart{\rho}$ is not already multicoloured by $\phi$, $e$ is multicoloured by exactly one of $\omega$ and $\overline \omega$. Therefore, the average
size of $\omega$ and $\overline \omega$, as 2-cuts of $H$, is
\begin{equation}\label{eq:excess-transfer}
z^U+z'+\left(e\left(H_{\rho}'\right)-z'\right)/2+N^\multi_\rho =\zpart/2+e\left(H_{\rho}'\right)/2+N^\multi_\rho.
\end{equation}
Now, let $\omega:V\to\left\{ 1,2\right\} $ be a uniformly random 2-cut of $H$, let $Z$ be the size of this 2-cut, and let $\Zpart$ be the size of $\omega$ restricted to $U$ as a 2-cut of $\Hpart{P\omega}$. Taking the expectation of \cref{eq:excess-transfer,} we have
\begin{equation} \E\left[Z\cond P\omega=\rho\right]=\E\left[\Zpart\cond P\omega=\rho\right]/2+e\left(H_{\rho}'\right)/2+N^\multi_\rho.\label{eq:excess-transfer-random}\end{equation}
Conditioning on the event $P\omega=\rho$, note that the restriction of $\omega$ to $U$
is a uniformly random 2-cut of $\Hpart{\rho}$, so $\E\left[\Zpart\cond P\omega=\rho\right]$ is the average size of a random 2-cut of $\Hpart{\rho}$. Combining \cref{eq:excess-transfer} and \cref{eq:excess-transfer-random}, if $\Hpart{\rho}$
has a 2-cut with some excess $x'$, then $H$ has a 2-cut with excess
at least
\begin{equation}
\left(x'+\E\left[\Zpart\cond P\omega=\rho\right]\right)/2+e\left(H_{\rho}'\right)/2+N^\multi_\rho-\E Z=x'/2+\E\left[Z\cond P\omega=\rho\right]-\E Z.\label{eq:excess-transfer-clean}
\end{equation}
Finally, observe that $\Hpart{P\omega}$ always contains (two copies of) all $\Omega(m)$ edges of $H[U]$ with size at least four and that there is an outcome $\rho$ of $P\omega$ such
that $\E\left[Z\cond P\omega=\rho\right]\ge \E Z$. Since $U=V(\Hpart{\rho})$, by assumption, $\Hpart{\rho}$ has a 2-cut with excess $\Omega(\sqrt p m/\sqrt \Delta)$, corresponding to a 2-cut of $H$ with at least half this excess.
\end{proof}

Now, our proof of \cref{lem:almost-linear-2} will follow the same general approach as for
\cref{lem:almost-linear-3}, but with a few added complications. Before we discuss the proof further, we give some generalisations of the lemmas
in \cref{subsec:3-cuts-almost-linear}. First, we generalise the notion of goodness from \cref{def:good-3}.

\begin{definition}
\label{def:good}Consider a mixed $k$-multigraph $H$ and a (not necessarily spanning) subgraph $H'$ of $H$. Consider a partition $\cV=\left(V_{1},\dots,V_{t}\right)$
of $V(H')$. We say that $\cV$ is \emph{$\left(m',\Delta'\right)$-good}
with respect to $\left(H,H'\right)$ if: 
\begin{enumerate}
\item[(i)] $\bigcup_{i}G\left(H'\right)\left[V_{i}\right]$ has
at least $m'$ edges; 
\item[(ii)] $\bigcup_{i}G\left(H\right)\left[V_{i}\right]$ has maximum degree
at most $\Delta'$; 
\item[(iii)] each edge $e$ of $H$ has vertices in at least $|e\cap V(H')|-1$ different
parts of $\cV$; 
\item[(iv)] there is no pair of edges $e_{1},e_{2}\in E(H)$ and part $V_i$ such that both $e_{1}$ and $e_{2}$ have two vertices in $V_{i}$, comprising a total of at least 3 vertices, and also $e_1$ and $e_2$ intersect in $V(H') \setminus V_{i}$. 
\end{enumerate}
We say that $\cV$ is \emph{$y$-almost }$\left(m',\Delta'\right)$-good
with respect to $(H,H')$ if properties (i) and (ii) hold, and if at most $y/2$ edges
of $H$ fail to satisfy property (iii) and at most $y/2$ pairs of
edges of $H$ fail to satisfy property (iv).
\end{definition}

Note that for a 3-multigraph $H$, if $\cV$ is ($y$-almost) $\left(m',\Delta'\right)$-good with respect to $(H,H[U])$ then $\cV$ is ($y$-almost) $\left(m',\Delta'\right)$-good with respect to $(H,U)$ as defined in \cref{def:good-3}. Therefore, the following lemma is a generalisation of \cref{lem:exists-good-3}.

\begin{lemma}
\label{lem:exists-good}Fix $k$ and consider an $m$-edge, $n$-vertex mixed $k$-multigraph
$H$. Let $H'$ be a (not necessarily spanning) subgraph of $H$ with $\Omega\left(m\right)$ edges. Suppose that $\deg_H\left(u\right)\le\Delta$
for all $u\in V(H')$ and $\deg_H\left(u,v\right)\le g$ for all distinct $u,v\in V(H')$.
Let $p=\Delta^{-3/5}\land\left(g^{-2/3}\Delta^{-1/3}\right)$ and
suppose $g=o( pm)$ and $g\log n=o(p\Delta)$. Then, for any $c>0$,
there is $m'=\Omega\left(pm\right)$ and $\Delta'=O\left(p\Delta\right)$
such that there exists a $\left(cm'/\sqrt{\Delta'}\right)$-almost
$\left(m',\Delta'\right)$-good partition $\cV=\left(V_{1},\dots,V_{t}\right)$
with respect to $\left(H,H'\right)$.
\end{lemma}

\begin{proof}
Let $t=1/(c'p)$ for some small $c'>0$ (we will assume that $c'$ is small enough to satisfy certain inequalities later in the proof). Let $\cV=\left(V_{1},\dots,V_{t}\right)$ be a random partition
of $V(H')$ obtained by choosing the part of each vertex uniformly at random, so that the probability a vertex falls into a particular part is $p':=c'p$. We will show that $\cV$ satisfies
the necessary almost-goodness properties with positive probability.

Let $M$ be the number of edges in $\bigcup_{i}G\left(H'\right)\left[V_{i}\right]$.
Then we have $\E M=p'e\left(G\left(H'\right)\right)$
and $\E M^{2}\le\left(\left(\vphantom{|^1} p'e\left(G\left(H'\right)\right)\right)^{2}+p'q\right)$,
where $q=O\left(mg\right)$ is the number of pairs of (possibly not
distinct) edges in $G(H')$ on the same two
vertices. (Note that for any two edges of $G(H')$ which do not coincide, their presence in $\bigcup_{i}G\left(H'\right)\left[V_{i}\right]$ is independent). Since we are assuming $g=o( pm)$, we have $\Var M=o\left(\left(\E M\right)^{2}\right)$,
so, by Chebyshev's inequality, property (i) holds a.a.s.\ with $m'=p'e\left(H'\right)/2$.

For each vertex $v\in V(H')$ and $1\le i\le t$, let $Q_v^i=\sum_{w\ne v} \deg_H(v,w)\one_{w\in V_i}$ be the number of edges of $G(H)$ between $v$ and a vertex of $V_i$. Note that $\E Q_v^i\le p'k\Delta$ and
$$\sum_{w\ne v} \Var\left(\deg_H(v,w)\one_{w\in V_i}\right)\le p'\sum_{w\ne v}\deg_H(v,w)^2=O(p\Delta g),$$
and note that each $|\deg_H(v,w)\one_{w\in V_i}|$ is bounded by $g$. If we set $\Delta'=2p'k\Delta=\Theta(p\Delta)$, Bernstein's inequality (see, for example,~\cite[Equation~(2.10)]{BLM13}) implies that,
$$\Pr(Q_v^i> \Delta')\le \exp\left(-\frac{(\Delta')^2}{O(p\Delta g+g\Delta')}\right)=e^{-\Omega(p\Delta/g)}=o(1/n^2).$$
(Here we have used the assumption that $g\log n=o( p\Delta)$.)
Therefore, by the union bound, a.a.s.\ $Q_v^i\le\Delta'$ for each $v,i$. That is, property (ii) a.a.s.\ holds.

Next, the number of pairs of edges of $H$ which share multiple vertices
in $V(H')$ is $O\left(mg\right)$. The probability that such a pair violates
property (iv) is $O\left(\left(p'\right)^{2}\right)$. The number
of pairs of edges of $H$ which share only one vertex in $V(H')$
is $O\left(m\Delta\right)$ and the probability that such a pair violates
property (iv) is $O\left(\left(p'\right)^{3}\right)$. Therefore, recalling the definition of $p$,
the expected number of pairs of edges violating property (iv) is 
\[
O\left(mg\left(p'\right)^{2}+m\Delta\left(p'\right)^{3}\right)\le c\left(m'/\sqrt{\Delta'}\right)/8
\]
for $c'$ sufficiently small in terms of $c$. So, with probability
at least $3/4$, the number of such edges is at most $c\left(m'/\sqrt{\Delta'}\right)/2$.

Finally, we consider (iii). The expected number of edges $f$ of $H$ which do not have vertices
in at least $|f\cap V(H')|-1$ different parts of $\cV$ is $O\left(\left(p'\right)^{2}m\right)\le c\left(m'/\sqrt{\Delta'}\right)/8$ (this inequality is not tight). Therefore, with probability
at least $3/4$, the number of edges violating property (iii) is at
most $c\left(m'/\sqrt{\Delta'}\right)/2$.
\end{proof}

We will also need a generalisation of \cref{lem:combine-cuts}. In \cref{lem:combine-cuts} we only considered multigraphs, but here we will need to consider higher-uniformity multihypergraphs. Unfortunately, it is not in general true that large-excess 2-cuts of induced subgraphs of a multihypergraph $H$ can be combined into a large-excess 2-cut of $H$. We therefore introduce the notion of \emph{average excess}, which does allow us to combine multiple cuts under certain conditions, as follows. For a subset $V'$ of the vertices of a multihypergraph $H$,
a \emph{$V'$-partial} 2-cut of $H$ is an assignment of parts
to the vertices in $V'$. The \emph{average size} of a $V'$-partial
2-cut $\omega':V'\to \{1,2\}$ is the expected size of the random 2-cut of $H$ obtained by starting with $\omega'$ and then choosing
the part of each vertex $v\notin V'$ uniformly at random. The average excess of $\omega'$ is its average size minus the expected size of a uniformly random cut of $H$.

\begin{lemma}
\label{lem:combine}Let $H$ be a mixed multihypergraph and let $V_{1},\dots,V_{t}$ be disjoint subsets of its vertex set such that each edge $f$ of $H$ has vertices in at least $|f|-1$ different parts. Suppose that for each $i\le t$ there is a
$V_{i}$-partial 2-cut of $H$ with average excess $x_{i}$. Then
there is a 2-cut of $H$ with excess at least $\sum_{i}x_{i}$.
\end{lemma}

\begin{proof}
As in the proof of \cref{lem:combine-cuts}, we can assume $V_1,\dots,V_t$ actually form a partition of $V(H)$, adding singleton parts if necessary. For each $1\le i\le t$, let $H_{i}\subseteq H$
be the multihypergraph of edges which intersect $V_{i}$ in two vertices and let $H_{0}$ contain the edges not in any other
$H_{i}$. By assumption, no edge of $H$ can appear in two different $H_{i}$. 

For each $i\in \{1,\dots,t\}$, let $\omega_{i}:V_{i}\to\left\{ 1,2\right\} $ be a
$V_{i}$-partial 2-cut with average excess $x_{i}$. For each edge $e$
in some $H_{i}$, let $p_{e}$ be the probability that that edge is
multicoloured in the random 2-cut obtained by starting with $\omega_{i}$
and choosing the part of each vertex $v\notin V_{i}$ uniformly at
random. Let $p_{e}^{*}=1-2^{1-|e|}$ be the probability that edge is multicoloured
in a uniformly random 2-cut. By the definition of average excess, for each $i$, we have
\[
\sum_{e\in E\left(H_{i}\right)}\left(p_{e}-p_{e}^{*}\right)=x_{i}.
\]
Now let $\sigma_{i}:\left\{ 1,2\right\} \to\left\{ 1,2\right\} $
be a uniformly random permutation and define the random 2-cut $\omega:V\left(H\right)\to\left\{ 1,2\right\} $
by $\omega\left(v\right)=\sigma_{i}\left(\omega_{i}\left(v\right)\right)$
for $v\in V_{i}$. Crucially, for each $e\in E(H_i)$ with $i\ne 0$, the restriction of $\omega$ to $e\setminus V_i$ is a uniformly random function $e\setminus V_i\to\{1,2\}$, because each vertex of $e\setminus V_i$ is in a different $V_j$. Therefore, the expected size of $\omega$ is
\[
\sum_{i\ne0}\sum_{e\in E\left(H_{i}\right)}p_{e}+\sum_{e\in E\left(H_{0}\right)}p_{e}^{*},
\]
so its expected excess is
\[
\sum_{i\ne0}\sum_{e\in E\left(H_{i}\right)}\left(p_{e}-p_{e}^{*}\right)=\sum_{i}x_{i}.
\]
Therefore, there is an outcome of $\omega$ with at least this excess.
\end{proof}

In order to work with the notion of average excess, we will need an additional lemma allowing us to reduce a large-average-excess problem to a large-excess problem for a certain auxiliary edge-weighted graph. We define the size and excess of a cut of an edge-weighted multigraph in the obvious way: the size of a cut is the sum of the weights of multicoloured edges and the excess of a cut is its size minus the expected size of a uniformly random cut.

\begin{lemma}
\label{lem:weighted-reduction}
Consider a multihypergraph $H$ and a vertex subset $V'\subseteq V(H)$ such that no edge of $H$ intersects $V'$ in more than two vertices. Consider the weighted graph $R$ on the vertex set $V'$, defined such that the weight between a pair of vertices $u,v$ is $$\eta_{u,v}=\sum_e 2^{-\left|e\setminus V'\right|}=\sum_e 2^{2-\left|e\right|},$$
where the sum is over all edges $e\in E(H)$ intersecting $V'$ in $u$ and $v$.
Then for any $\omega:V'\to \{1,2\}$, the excess of $\omega$ as a 2-cut of $R$ is exactly equal to the average excess of $\omega$ as a $V'$-partial 2-cut of $H$.
\end{lemma}

\begin{proof}
Let $z$ be the size of $\omega$ as a 2-cut of $R$ and let $z'$ be the average size of $\omega$ as a $V'$-partial 2-cut of $H$. It suffices to show that $z-z'$ does not depend on $\omega$.

First, consider the edges $e$ intersecting $V'$ in at most one vertex. The contribution of each such $e$ to $z$ is zero and the contribution to $z'$ is $p_{e}^{*}=1-2^{1-|e|}$, independently of $\omega$.

Next, for each $u,v$ let $N_{u,v}$ be the number of terms in the sum defining $\eta_{u,v}$. If $\omega(u)=\omega(v)$ then for any edge $e\in E(H)$ intersecting $V_{i}'$ in $u$ and $v$, the probability that $e$ is not multicoloured in a uniformly random extension of $\omega$ is $2^{2-\left|e\right|}$. So, the contribution of all such edges to $z'$ is $N_{u,v}-\eta_{u,v}$, while the contribution to $z$ is zero. On the other hand, if $\omega(u)\ne\omega(v)$, then the contribution of $\{u,v\}$ to $z'$ is $N_{u,v}$, while the contribution to $z$ is $\eta_{u,v}$.
\end{proof}

Next, we give the necessary generalisation of \cref{lem:random-subgraph-cut}. The \emph{kurtosis} of a random variable $X$ is $\E(X-\E X)^4/(\Var X)^2$.

\begin{lemma}
\label{lem:random-cut}
Let $F$ and $G$ be multigraphs on a set $V$ of
$n$ vertices. Suppose $G$ has $m$ edges and maximum degree $\Delta$ and let $g_{u,v}$ be the number of edges between two vertices $u$ and $v$ in $G$.
Consider a random weighted graph $R$ on the same vertex set $V$, where the weight between vertices $u$ and $v$ is a random variable $\eta_{u,v}$ with variance $\Omega(g_{u,v})$ and kurtosis $O(1)$. Suppose, moreover, that the weights $\eta_{u,v}$ are independent. Let $X^{\max}$ be the maximum excess of a 2-cut of $F\cup R$. Then $\E X^{\max}=\Omega\left(m/\sqrt{\Delta}\right)$.
\end{lemma}

Before proving \cref{lem:random-cut} we briefly explain why \cref{lem:random-subgraph-cut} follows from it. For this we will need the simple observation that sums of independent low-kurtosis random variables themselves have low kurtosis.

\begin{lemma}
\label{lem:kurtosis-sum}
Let $X_1,\dots,X_n$ be independent random variables each with kurtosis $O(1)$. Then $X=\sum_i X_i$ has kurtosis $O(1)$.
\end{lemma}
\begin{proof}
We may suppose without loss of generality, that $X_{1},\dots,X_{n}$ have
mean zero. Let $\sigma_{i}^{2}=\Var X_{i}=\E X_i^2$ and
let $\sigma^{2}=\sum_{i}\sigma^2_{i}=\Var X$. By the independence of the $X_i$ and the kurtosis assumptions, we have
\[\E X^4=\sum_i \E X_i^4+6\sum _{i\ne j} \sigma_i^2 \sigma_j^2 =O(\sigma^4).\tag*{\qedhere}\]
\end{proof}

\begin{proof}[Proof of \cref{lem:random-subgraph-cut} given \cref{lem:random-cut}]
Recall the definition of the random multigraph $R$ arising from $G$ in the statement of \cref{lem:random-subgraph-cut}. The number of edges $\eta_{u,v}$ in $R$ between any two vertices $u,v$ can be represented in the form $\eta_{u,v}=\eta^{(1)}_{u,v}+\dots+\eta^{(q)}_{u,v}$, where the $\eta^{(i)}_{u,v}$ are independent, and each $\eta^{(i)}_{u,v}$ satisfies $\Pr(\eta^{(i)}_{u,v}=g^{(i)}_{u,v})=p$ and $(\eta^{(i)}_{u,v}=0)=1-p$ for some $g^{(1)}_{u,v},\dots,g^{(q)}_{u,v}\in \NN$ with $g^{(1)}_{u,v}+\dots+g^{(q)}_{u,v}=g_{u,v}$. Note that $\Var(\eta^{(i)}_{u,v})=\left(g^{(i)}_{u,v}\right)^{\!2}p(1-p)=\Omega(g^{(i)}_{u,v})$, so $\Var \eta_{u,v}=\Omega(g_{u,v})$. Also, one may compute that the kurtosis of each $\eta^{(i)}_{u,v}$ is $(1-3p+3p^2)/(p-p^2)=O(1)$, so by \cref{lem:kurtosis-sum}, $\eta_{u,v}$ has kurtosis $O(1)$ as well. We may therefore apply \cref{lem:random-cut}.
\end{proof}

We will also need an additional lemma showing that for sums
of independent low-kurtosis random variables we are fairly likely to see fluctuations of size comparable to the standard deviation.

\begin{lemma}
\label{lem:sum-kurtosis}
Let $X_{1},\dots,X_{n}$ be a sequence of independent
random variables with kurtosis $O(1)$ and let $X=\sum_{i}X_{i}$. Then there is a positive constant $c$ (not
depending on $n$) such that 
\[
\Pr\left(\left|X-\E X\right|\ge c\sqrt{\Var X}\right)\ge c.
\]
\end{lemma}

\begin{proof}
Recalling \cref{lem:kurtosis-sum}, it suffices to show that the desired conclusion holds whenever $X$ itself has kurtosis $O(1)$. Also, without loss of generality, it suffices to consider the case where $\E X=0$. The desired conclusion now follows from \cite[Proposition~9.4]{Odo14}, but we include the short proof for completeness. By the Paley--Zygmund inequality,
\begin{align*}
\Pr\left(|X|\ge \frac12 \sqrt{\Var X}\right) = \Pr\left(X^2\ge \frac1{4} \E X^2\right) \ge \left(1-1/4\right)^2\frac{\left(\E X^2\right)^2}{\E X^4}=\Omega(1).\tag*{\qedhere}
\end{align*}
\end{proof}

Now we prove \cref{lem:random-cut}.
\begin{proof}[Proof of \cref{lem:random-cut}]
Consider a random ordering of the vertices of $G$ and let $d_{<}\left(v\right)$
be the number of edges between $v$ and the vertices that precede
it (in $G$). For each $v$, because a random ordering has the same
distribution as its reverse, note that $d_{<}\left(v\right)$ has
the same distribution as $d_{G}\left(v\right)-d_{<}\left(v\right)$,
so
$\E\sqrt{d_{<}\left(v\right)}=\E\left[(1/2)\left(\sqrt{d_{<}\left(v\right)}+\sqrt{d_{G}\left(v\right)-d_{<}\left(v\right)}\right)\right]$. By concavity, $\sqrt{d_{<}\left(v\right)}+\sqrt{d_{G}\left(v\right)-d_{<}\left(v\right)}\ge \sqrt{d_{G}\left(v\right)}$, so in fact $\E\sqrt{d_{<}\left(v\right)}\ge \sqrt{d_{G}\left(v\right)}/2$. Hence, we can fix an ordering with 
\[
\sum_{v\in V}\sqrt{d_{<}\left(v\right)}\ge\sum_{v\in V}\frac{\sqrt{d_{G}\left(v\right)}}2\ge \sum_{v\in V}\frac{d_G(v)}{2\sqrt \Delta}=\Omega\left(\frac{m}{\sqrt{\Delta}}\right).
\]
We now build a cut with parts $A$ and $B$ greedily, as follows. For each vertex $v$
in the chosen order in turn, expose the edges of $R$ between $v$ and
the preceding vertices. Consider the total weight $e_{A}\left(v\right)$
(respectively, $e_{B}\left(v\right)$) of edges between $v$ and the vertices
previously placed in part $A$ (respectively, part $B$) in $F\cup R$.
If $e_{A}\left(v\right)\ge e_{B}\left(v\right)$, put $v$ in
part $B$. Otherwise, put $v$ in part $A$. The outcome of this procedure
is a cut with excess at least $\sum_{v}\left|e_{A}\left(v\right)-e_{B}\left(v\right)\right|/2\le X^{\max}$.

For each step $v$, let $A_v$ be the set of vertices already placed in $A$ and let $B_v$ be the set of vertices already placed in $B$. Conditioning on these previous choices, note that
$$e_{A}\left(v\right)-e_{B}\left(v\right)=\sum_{u\in A_v}\eta_{u,v}-\sum_{u\in B_v}\eta_{u,v}$$ is a sum of independent random variables each with kurtosis $O\left(1\right)$ and the variance of this sum is
$$\sum_{u\in A_v}\Omega(g_{u,v})+\sum_{u\in B_v}\Omega(g_{u,v})=\Omega\left(d_{<}\left(v\right)\right).$$By \cref{lem:sum-kurtosis}, there is $c=\Omega\left(1\right)$ such that 
\[
\Pr\left(\left|e_{A}\left(v\right)-e_{B}\left(v\right)\right|\ge c\sqrt{d_{<}\left(v\right)}\right)\ge c.
\]
This means that 
\[
\E X^{\max}\ge \E\sum_{v\in V}\frac{\left|e_{A}\left(v\right)-e_{B}\left(v\right)\right|}2\ge \frac{c^{2}}2\sum_{v\in V}\sqrt{d_{<}\left(v\right)}=\Omega\left(\frac{m}{\sqrt{\Delta}}\right),
\]
as desired.
\end{proof}
We need one more ingredient to prove \cref{lem:almost-linear-2}. The following lemma was first proved by Bonami~\cite{Bon68} (for a short proof see~\cite{Odo14}, where it is given the name ``Bonami Lemma''). It can be viewed as a simple case of the Bonami--Beckner hypercontractive inequality for Boolean functions.

\begin{lemma}
\label{lem:bonami}
Let $\xi_1,\dots,\xi_n$ be independent Rademacher random variables, that is $\Pr(\xi_i=1)=\Pr(\xi_i=-1)=1/2$. Let $X$ be a degree-$b$ polynomial in the variables $\xi_1,\dots ,\xi_n$. Then $\E X^4\le 9^b (\E X^2)^2$.
\end{lemma}

Now, we can finally prove \cref{lem:almost-linear-2}, recalling from \cref{lem:no-bad-reduction} that it suffices to consider the case $U=V(H)$ (we will therefore not mention $U$ in the proof).
\begin{proof}[Proof of \cref{lem:almost-linear-2}]
Let $V=V(H)$ and, for $W\subseteq V$, let $\overline W=V\setminus W$. As in the proof of \cref{lem:almost-linear-3}, we need some way to introduce randomness into the problem. For \cref{lem:almost-linear-3} we exposed one of the three parts of a random cut, but here in the 2-cut case we have no parts to spare. Our approach is to choose some subset $W\subseteq V$ and expose the parts of the vertices outside $W$. Using the same averaging trick as in \cref{lem:no-bad-reduction}, this information will allow us to define a random hypergraph $\Hpart{}$ on the vertex set $W$, such that a large cut of $\Hpart{}$ yields a large cut of $H$.

Specifically, first define the ``information restriction'' operator $P$ as follows. For $\omega:V\to\{1,2\}$, let $P\omega:\overline W\to \{1,2\}$ be the restriction of $\omega$ to $\overline W$. Then, for $\omega:V\to\{1,2\}$, let $H_{P\omega}'\subseteq H$ consist of
those edges $e$ in $H$ but not in $H[W]$ with $\left|P\omega\left(e\setminus W\right)\right|<2$ and $\left|e\cap W\right|>0$, and let $\Hpart{P\omega}$ be the mixed $k$-multigraph with vertex
set $W$, two copies of every edge in $H[W]$, and an edge $e\cap W$ for each $e\in E(H_{P\omega}')$. This definition is as in \cref{lem:no-bad-reduction}, but with $W$ in place of $U$.

If $\omega$ is a uniformly random 2-cut, then $\Hpart{P\omega}$ is a random hypergraph. In particular, for any edge $e$ of $H$ that intersects $W$ and $\overline W$, that edge will contribute to $\Hpart{P\omega}$ with probability $2^{1-|e\setminus W|}$. If $|e\setminus W|=1$, then this event is not actually random and if $|e\cap W|=1$, then the corresponding edge in $\Hpart{P\omega}$ will have just one vertex. So, the number of edges of $H$ which meaningfully contribute a random edge to $\Hpart{P\omega}$ is of the same order of magnitude as the number of edges of $H$ which intersect both $W$ and $\overline W$ in at least two vertices. For a typical choice of $W$, this is of the same order of magnitude as the number of edges of $H$ which have size at least $4$, which explains why we need this quantity to be large.

We will later go into more detail about our choice of $W$ and the properties of $\Hpart{P\omega}$, but the basic idea is that it suffices to show that the random hypergraph $\Hpart{P\omega}$ has a large-excess 2-cut. The first step towards this goal will be to use \cref{lem:exists-good} to obtain a partition $V_1',\dots,V_t'$ of $W$ satisfying the condition in the statement of \cref{lem:combine}. Then, let $X_{i}$ be the maximum average excess of a $V_{i}'$-partial
2-cut of $\Hpart{P\omega}$, and let $X=\sum_{i}X_i$. By \cref{lem:combine}, $\Hpart{P\omega}$ has a 2-cut of excess $X$. As in the proof of \cref{lem:almost-linear-3}, we may therefore turn our attention towards estimating the $\E X_i$.

Towards this end, we will use \cref{lem:weighted-reduction} to interpret $X_i$ as the maximum excess of a 2-cut in a certain weighted graph $R_{i}$ on the vertex set $V_i'$. Since $\Hpart{P\omega}$ is random, this weighted graph $R_i$ will be random as well, and in fact, due to certain properties of the $V_i'$, the edges of $R_i$ will be independent. We will finally apply \cref{lem:random-cut} to $R_i$ to estimate $\E X_i$. In particular, we will verify the necessary kurtosis condition with \cref{lem:bonami}.

Now we turn to the details of the proof, which are presented in a slightly different order to the outline above. First, we apply \cref{lem:exists-good} to obtain a suitable good partition. Let $H^{(\ge4)}\subseteq H$ be the subgraph of edges with size at least four (with vertex set $V$). Apply \cref{lem:exists-good} (with some small $c=\Omega(1)$) to obtain a $cx$-almost $\left(2m',\Delta'\right)$-good
partition $\cV=\left(V_{1},\dots,V_{t}\right)$ with respect to $\left(H,H^{(\ge4)}\right)$,
where $m'=\Omega(pm)$, $\Delta'=O(p\Delta)$ and $x=m'/\sqrt{\Delta'}$. As in the proof of \cref{lem:almost-linear-3}, delete $cx$ edges from $H$ (causing corresponding changes in $H^{(\ge 4)}$) so that
$\cV$ is an $\left(m',\Delta'\right)$-good partition with respect to $\left(H,H^{(\ge4)}\right)$. We will show that the resulting multihypergraph $H$ has a 2-cut with excess $\Omega(x)$, where the implied constant does not depend on $c$. This will suffice to prove \cref{lem:almost-linear-2}.

Next we choose $W$. Start by letting $W$ be a random set of vertices obtained
by including each vertex independently with probability $1/2$. For each $i$, let $H_{i}$ be the subgraph of $H^{(\ge4)}$ consisting of all edges that intersect $V_{i}$ in two vertices. By properties (i) and (iii) of \cref{def:good}, the $H_i$ are disjoint and $\sum_i e(H_i)\ge m'$. Now, let $G_i$ be the random multigraph with vertex set $V_i\cap W$ and an edge $\{u,v\}$ for every edge of $H_i$ containing $u,v$ and at least two vertices outside $W$. Each edge $e$ of $H_i$ contributes to $G_i$ with probability $\Omega(1)$, so $\E e(G_i)=\Omega(e(H_i))$ and we can fix an outcome of $W$ with
\begin{equation}
\sum_ie(G_i)=\sum_{i}\Omega\left(e(H_i)\right)=\Omega\left(m'\right)\label{eq:eGi}.
\end{equation}
This outcome of $W$ will be fixed for the rest of the proof, so the $G_i$ are now fixed graphs. Let $V_i'=V_i\cap W$.

Now, let $\omega:V\to\left\{ 1,2\right\} $ be a uniformly random
2-cut of $H$, and let $Z$ be the size of this 2-cut. Recall the definition of $P$ and $\Hpart{P\omega}$ from the outline at the start of the proof. As in the proof of \cref{lem:no-bad-reduction}, if, for some outcome $\rho$ of $P\omega$, $\Hpart{\rho}$ has a 2-cut $\phi$ with some excess $x'$, then we may consider the average size of two different cuts related to $\phi$ to show that $H$ has a 2-cut with excess
at least
\begin{equation}x'/2+\E\left[Z\cond P\omega=\rho\right]-\E Z.\label{eq:excess-transfer-2}\end{equation}
(The proof of \cref{eq:excess-transfer-2} is exactly the same as the proof of \cref{eq:excess-transfer-clean}, replacing ``$U$'' with ``$W$''.)

Recall that we defined $X_{i}$ to be the maximum average excess of a $V_{i}'$-partial
2-cut of $\Hpart{P\omega}$ and set $X=\sum_{i}X_i$. As in the proof of \cref{lem:almost-linear-3}, it suffices to show that $\E X=\Omega\left(x\right)$. Indeed, by \cref{lem:combine} and property (iii) of \cref{def:good}, the multihypergraph $\Hpart{P\omega}$ has a cut with excess $X$, and by \cref{eq:excess-transfer-2}, it follows that $H$ has a cut with excess at least $X/2+E\left[Z\cond P\omega\right]-\E Z$. If we could prove that $\E X=\Omega\left(x\right)$ we could conclude that $H$ has a cut with excess at least
$$\E \left[X/2+\E\left[Z\cond P\omega\right]-\E Z\right]=\Omega(x),$$ as desired.

We now wish to estimate each $\E X_i$. Let $R_i$ be the random edge-weighted graph such that the weight between any pair of distinct vertices $u,v\in V_{i}'$ is the random variable
$$\eta_{u,v}=\sum_f 2^{2-\left|f\right|},$$ 
where the sum is over all edges $f\in E(\Hpart{P\omega})$ intersecting $V_{i}'$ in $u$
and $v$. By \cref{lem:weighted-reduction}, $X_i$ is the maximum excess of a 2-cut of $R_i$.

Now, recalling the definition of $\Hpart{P\omega}$, the direct interpretation of $\eta_{u,v}$ (in terms of the random outcome of $P\omega$) is as follows. An edge $e\in H$ with $e\not\subseteq W$ and $u,v\in e$ contributes a term $2^{2-\left|e\cap W\right|}$ to $\eta_{u,v}$ if $|P\omega(e\setminus W)|<2$ (this event is only random if $|e\setminus W|\ge2$, meaning that $e$ is in $H_i\subseteq H^{(\le 4)}$). Edges $e\in H[W]$ with $u,v\in e$ contribute two non-random terms both equal to $2^{2-|e|}$ (because we included two copies of each such edge in $\Hpart{P\omega}$).

An immediate consequence of the above discussion is that the $\eta_{u,v}$ are independent. Indeed, by (iv) of \cref{def:good}, if $\{u_1,v_1\}\ne \{u_2,v_2\}$ are different pairs of vertices in $V_i'$, $e_1$ is an edge of $H$ containing $u_1$ and $v_1$ and $e_2$ is an edge of $H$ containing $u_2$ and $v_2$, then $e_1\setminus V_i'\supseteq e_1\setminus W$ and $e_2\setminus V_i'\supseteq e_2\setminus W$ do not intersect. However, the edges of $H$ containing $u$ and $v$ may intersect each other in complicated ways, so the $\eta_{u,v}$ may have rather complicated distributions. We next need to estimate the variance and kurtosis of these distributions. In particular, to apply \cref{lem:random-cut} we need to show that each $\eta_{u,v}$ has kurtosis $O(1)$ and we estimate the variances of the $\eta_{u,v}$ in terms of the edge multiplicities of $G_i$. Recall that $G_i$ has an edge $\{u,v\}$ for every edge of $H_i$ containing $u,v$ and at least two vertices outside $W$. For each pair of distinct vertices $u,v\in V_{i}$, let $g_{u,v}$ be the number of edges between $u$ and $v$ in $G_{i}$.

\begin{claim*}
For each pair of distinct vertices $u,v\in V_{i}$, $\eta_{u,v}$
has kurtosis $O(1)$ and $\Var\eta_{u,v}=\Omega\left(g_{u,v}\right)$.
\end{claim*}
\begin{proof}
For $u,v\in V_i'$, let $H_{u,v}$ be the set of all $g_{u,v}$ edges
$e\in E\left(H^{(\ge4)}\right)$ containing $u$, $v$ and at least two vertices
outside $W$. For each $e\in E(H_{u,v})$, define
\[
Q_{e}=2^{2-\left|e\cap W\right|}\one_{\left|P\omega\left(e\setminus W\right)\right|<2}.
\]
Observe that $\eta_{u,v}$ is the sum of all such $Q_{e}$ (strictly speaking $\eta_{u,v}$ is a translation of this sum: the edges which have fewer than two vertices outside $W$ each contribute additional non-random terms).

If $e_1\setminus W$ and $e_2\setminus W$ are disjoint, then $Q_{e_1}$ and $Q_{e_2}$ are independent. Otherwise, note that for any $h\subseteq \overline W$ we have $\Pr(|P\omega(h)|<2)=2^{1-|h|}$, so $$\Cov\left(Q_{e_{1}}, Q_{e_{2}}\right)=2^{2-\left|e_1\cap W\right|}2^{2-\left|e_2\cap W\right|}(2^{1-|e_1\cup e_2\setminus W|}-2^{1-|e_1\setminus W|}2^{1-|e_2\setminus W|}).$$ That is, $\Cov\left(Q_{e_{1}}, Q_{e_{2}}\right)=0$ unless $e_{1}\setminus W$ and $e_{2}\setminus W$ share
at least two vertices, in which case $\Cov\left(Q_{e_{1}}, Q_{e_{2}}\right)=\Theta(1)$. This implies that $\Var\left(\eta_{u,v}\right)=\Theta\left(N_2\right)$,
where $N_2\ge g_{u,v}$ is the number of ordered pairs of (not necessarily
distinct) edges of $H_{u,v}$ intersecting in at least two vertices outside $W$.

Next, for each $w\in \overline W$, let $\xi_w=(-1)^{P\omega(w)}$, so that the $\xi_w$ are independent Rademacher random variables as in \cref{lem:bonami}. For each $h\subseteq \overline W$, let $y$ be an arbitrary vertex in $h$ and observe that
$$\one_{\left|P\omega\left(h\right)\right|<2}=\frac1{2^{|h|-1}}\prod_{w\in h\setminus\{y\}}(1+\xi_y\xi_w).$$
This means that each $Q_e$ is a polynomial in the $\xi_w$ with degree $O(1)$ (actually, using the fact that $\xi^2=1$ for $\xi\in \{-1,1\}$, the degree can be bounded by $|e\setminus W|\le k-2$). Therefore, $\eta_{u,v}-\E \eta_{u,v}$ is a polynomial in the $\xi_w$ with degree $O(1)$ as well and, by \cref{lem:bonami}, it follows that $\eta_{u,v}$ has kurtosis $O(1)$, as desired.
\end{proof}
By the above claim and property (ii) of \cref{def:good}, we may apply \cref{lem:random-cut} with $F=\emptyset$, $G=G_i$, $R=R_i$, $\Delta=\Delta'$ to see that
$$\E X_{i}=\Omega\left(\frac{e(G_i)}{\sqrt{\Delta'}}\right).$$
Recalling \cref{eq:eGi}, it follows that
\begin{align*}
\E\sum_{i}X_{i}&=\Omega\left(\frac{m'}{\sqrt{\Delta'}}\right)=\Omega(x),
\end{align*}
completing the proof.
\end{proof}

\section{Concluding remarks\label{sec:concluding}}

For fixed $2 \le r\le k$ with $k>3$ or $r>2$, we have shown that every $k$-graph has max-$r$-cut at least
\[
\frac{S\left(k,r\right)r!}{r^{k}}m+\Omega\left(m^{5/9}\right),
\]
while there exist $k$-graphs with max-$r$-cut only
\[
\frac{S\left(k,r\right)r!}{r^{k}}m+O\left(m^{2/3}\right).
\]
The most interesting problem left open by this paper is to close the gap between these bounds. We make the following
conjecture.
\begin{conjecture}
For fixed $2 \le k\le r$ with $k>3$ or $r>2$, every $k$-graph has an
$r$-cut of size
\[
\frac{S\left(k,r\right)r!}{r^{k}}m+\Omega\left(m^{2/3}\right).
\]
\end{conjecture}

We remark that our specific bound $\Omega(m^{5/9})$ is simply what arises from the tradeoff between the lemmas in \cref{subsec:almost-linear-reduction} and \cref{lem:almost-linear-3,lem:almost-linear-2}, and it's possible that further improvements could come from both directions. In particular, for almost-linear hypergraphs, our use of \cref{lem:exists-good} seems quite wasteful, as we are dividing our hypergraph into many small parts and only dealing with the small fraction of edges which intersect a part in two vertices. The reason we do this is to reduce our problem to a collection of random subproblems each with few dependencies, but one might hope that a more sophisticated argument could tolerate dependencies somehow.

We have also shown that  every $m$-edge, $n$-vertex
$k$-graph with no isolated vertices has an $r$-cut of size
\[
\frac{S\left(k,r\right)r!}{r^{k}}m+c_{r,k}n
\]
for some positive constant $c_{r,k}$. When $r = 2$, we were able to determine the best possible value for the constants $c_{2,k}$. It remains an open problem to do the same for $r > 2$.

Finally, we remark that in graphs there is a more general notion
of a cut than we discussed in the introduction. An \emph{$\ell$-cut}
is a partition of the vertex set into $\ell$ parts and the size
of such an $\ell$-cut is the number of edges which have both their
vertices in different parts. Extending this more general notion of
a cut to hypergraphs would result in a very general two-parameter
family of cut problems. Indeed, for $r \le \ell, k$, we can define
the \emph{$r$-size} of an $\ell$-cut to be the number of edges which have
vertices in at least $r$ of the $\ell$ parts and the \emph{max-$\left(r,\ell\right)$-cut} of a $k$-graph $H$ to
be the maximum $r$-size among all $\ell$-cuts of $H$. We imagine
that our methods could be fairly straightforwardly adapted to prove generalised
counterparts of \cref{thm:nontrivial,thm:excess-n}, but we have not 
explored this further.

\end{document}